\documentclass[11pt]{amsart}
\usepackage[margin=1.5in]{geometry}
\usepackage{amsmath,amssymb,amsthm,mathrsfs,bm}
\usepackage{graphicx}
\usepackage{xcolor}
\usepackage{tikz}
\usepackage{float}
\usepackage{algorithm}
\usepackage{placeins}
\usepackage{subcaption}
\usepackage{hyperref}
\usepackage{cleveref}
\usepackage{bbm}
\usepackage{comment}
\usepackage[
    backend=biber,
    style=numeric,
    doi=true,
    url=true,
    isbn=false
]{biblatex}

\usetikzlibrary{positioning,arrows.meta}

\title[ Data-Driven Stochastic Control of Delayed air-conditioning Systems ]{Data-Driven Stochastic Optimal Control of air-conditioning Systems with Delayed Activation}

\author{Christian Bayer}
\address[Christian Bayer]{Weierstraß Institute for Applied Analysis and Stochastics, Berlin, Germany}
\email{christian.bayer@wias-berlin.de}

\author{Raúl Tempone}
\address[Raúl Tempone]{Computer, Electrical and Mathematical Sciences \& Engineering Division (CEMSE), King Abdullah University of Science and Technology (KAUST), Thuwal, Saudi Arabia}
\email{raul.tempone@kaust.edu.sa, tempone@uq.rwth-aachen.de}

\author{Leon Wilkosz$^\ast$}
\address[Leon Wilkosz]{Computer, Electrical and Mathematical Sciences \& Engineering Division (CEMSE), King Abdullah University of Science and Technology (KAUST), Thuwal, Saudi Arabia} 
\email{leon.wilkosz@kaust.edu.sa}
\thanks{Corresponding author$^\ast$}

\date{\today}

\usepackage{aliascnt}
\usepackage{hyperref}
\usepackage{cleveref}

\newtheorem{theorem}{Theorem}

\newaliascnt{lemma}{theorem}

\aliascntresetthe{lemma}

\newaliascnt{proposition}{theorem}
\newtheorem{proposition}[proposition]{Proposition}
\aliascntresetthe{proposition}

\newaliascnt{remark}{theorem}

\aliascntresetthe{remark}

\newcommand{\R}{\mathbb R}

\DeclareMathOperator*{\argmin}{arg\,min}
\definecolor{darkred}{rgb}{0.55,0,0}

\begin{document}

\begin{abstract}
We consider a stochastic optimal control problem arising in the 
operation of industrial air-conditioning systems with delayed 
activation effects. The delayed transition between inactive and 
operational regimes induces a non-Markovian structure in the 
controlled dynamics. We derive a 
regression-based discrete dynamic programming algorithm on nested 
state-control grids. We give a probabilistic proof of convergence of the idealized
time-discrete dynamic programming problem to the continuous-time value function. 
Numerical experiments based on calibrated multidimensional 
forecast-tracking Ornstein--Uhlenbeck models illustrate the 
effectiveness of the proposed methodology.
\end{abstract}



\maketitle

\section{Introduction}

The operation of large-scale air-conditioning systems in industrial and scientific environments requires the reliable maintenance of strict environmental conditions over long time horizons. Typical examples include airport infrastructures, semiconductor manufacturing facilities, and pharmaceutical laboratories in which deviations from prescribed temperature and humidity regimes may result in substantial operational losses or production failure \cite{ni_semiconductor_hvac, yin_cleanroom_energy}. In such systems, heating, ventilation, and air-conditioning infrastructure constitutes one of the dominant sources of operational energy consumption, often accounting for a substantial fraction of the total facility energy demand \cite{perez_lombard_building_energy, zeng_hvac_review}. Simultaneously, failure to maintain the required environmental conditions may compromise product quality, interrupt sensitive production processes, or damage critical infrastructure, thereby inducing severe economic consequences \cite{yin_cleanroom_energy}. Consequently, the operation of industrial air-conditioning systems naturally gives rise to a sequential control problem balancing the energy costs associated with active cooling against the economic penalties caused by violations of prescribed environmental conditions.

A central difficulty in the mathematical modeling of such systems is the 
delayed response of the cooling infrastructure. Existing optimization and 
stochastic control approaches for HVAC systems typically assume that 
control decisions affect the environmental dynamics instantaneously, 
while uncertainty enters only through exogenous disturbances such as 
weather conditions or occupancy fluctuations \cite{yang_hvac_soc, kumar_hvac_mpc}. 
In practice, however, large-scale cooling systems frequently exhibit 
substantial activation delays. After being switched on, the air 
conditioning system does not immediately enter the operational regime, but 
instead undergoes a warm-up phase during which energy costs are already 
incurred while the environmental conditions may still violate admissible 
operating thresholds. Consequently, the resulting control problem exhibits 
an intrinsic delay structure. In particular, the future operational state of 
the system does not only depend on the present switching state, but additionally 
on the elapsed time since the most recent activation event. Delayed reaction and 
execution lags have been studied in stochastic impulse-control settings. 
The contribution here is an HVAC-oriented stochastic switching formulation 
that combines an explicit warm-up clock, weak compactness-based analysis, 
regression-based dynamic programming on the augmented state, and 
calibration to forecast-tracking environmental data.
\cite{ni_semiconductor_hvac, yin_cleanroom_energy}. 
As a result, the air-conditioning system operates in three 
regimes:

\begin{itemize}
\item Off,
\item Warm-up,
\item Operational,
\end{itemize}
and we consider two types of costs:
\begin{itemize}
\item operational cost, from now on called $C_1$,
\item failure to provide cooling/dehumidifying service when required, from now on called $C_2$.
\end{itemize}
The system requires a deterministic warm-up period $\tau$ before the admissible environmental regime is reached. We assume conservatively that whenever a switch to Off occurs, the system needs to go through the full warm-up time to enter the operational regime again.
We illustrate this in \Cref{fig:operational_states}. 
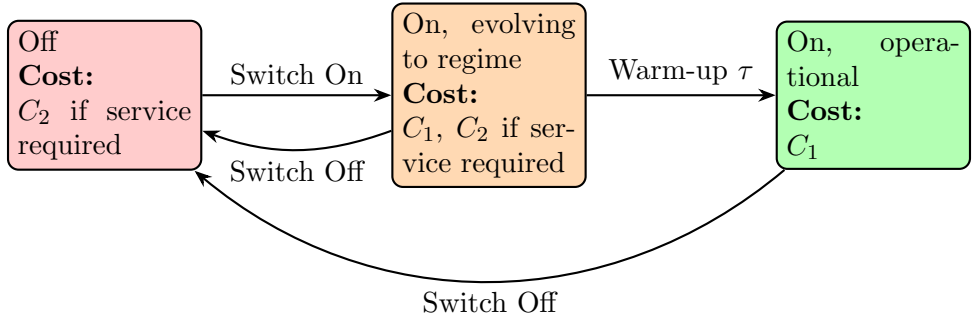
\begin{figure}[ht]
\centering
\begin{tikzpicture}[node distance=2.5cm, >=Stealth, auto, thick, scale=1]
    \node[draw, rounded corners, fill=red!20, minimum width=2.5cm, minimum height=0.8cm] (off)
        {\parbox{2.3cm}{Off\\ \textbf{Cost:} \\$C_2$ if service required}};
    \node[draw, rounded corners, fill=orange!30, minimum width=2.5cm, minimum height=0.8cm, right=of off] (on_evolve)
        {\parbox{2.3cm}{On, evolving to regime\\ \textbf{Cost:} \\$C_1$, $C_2$ if service required}};
    \node[draw, rounded corners, fill=green!30, minimum width=2.5cm, minimum height=0.8cm, right=of on_evolve] (on_regime)
        {\parbox{2.3cm}{On, operational\\ \textbf{Cost:} \\$C_1$}};
    \draw[->] (off) -- (on_evolve) node[midway, above] {Switch On};
    \draw[->] (on_evolve) -- (on_regime) node[midway, above] {Warm-up $\tau$};
    \draw[->, bend left=20] (on_evolve) to node[midway, below] {Switch Off} (off);
    \draw[->, bend left=40] (on_regime) to node[midway, below] {Switch Off} (off);
\end{tikzpicture}
\caption{Delayed switching structure of the air-conditioning system. After activation the system enters a warm-up phase during which running costs are incurred, while failure costs may still occur if service is required.}
\label{fig:operational_states}
\end{figure}
Due to unavailability of forecasts for the internal environmental parameters 
of buildings, and the complex dependence on architectural design and materials, 
we use the outside temperature at the building's location as a proxy for operational 
risk.
Thus, as a major modelling assumption, the process of environmental parameters 
$X_t \in \R^d, \; 0 \leq t \leq T$ does not depend on the decisions of the controller.
Let us assume for now that we are given a trajectory of $X$. We assume we are given
a set of infeasible states
\[
\mathcal{F} = \{x \in \R^d \mid  \text{service is required at x} \}.
\]
The controls $\alpha_t, \; 0 \leq t \leq T$ take the values $0$ (Off), and $1$ (On). Furthermore, for a given control
 path $\alpha$ we define a clock variable
\[
\zeta_t = t - \sup\{s \leq t \mid \alpha_s = 0\}
\]
measuring the time since the last Off switch. The cost associated with a given state/control pair trajectory is thus
\begin{align}\label{cost_rough}
    \begin{split}
    C(\alpha) = & \; C_1 \int_0^T \alpha_t dt + C_2 \int_0^T (1 - \alpha_t) \mathbbm{1}_{\mathcal{F}}(X_t) d t \\
    & \; + C_2 \int_0^T \alpha_t \mathbbm{1}_{\mathcal{F}}(X_t) \mathbbm{1}_{[0,\tau)}(\zeta_t) d t,
    \end{split}
\end{align}
which makes clear that an optimal closed-loop (i.e. $\mathcal{F}^{X}$-adapted) control $\alpha$ is not of feedback type (i.e. Markovian), as it must depend on path information through the clock $\zeta$.

The present work develops a mathematically rigorous framework for this class of delayed stochastic optimal switching problems. 
Sufficiently regular behaviour is enforced by adding a jump penalty to \Cref{cost_rough}.

On the numerical side, we construct a hybrid grid and regression-based dynamic programming algorithm.
The delay structure is incorporated through a clock variable, leading to a 
state representation suitable for discrete dynamic programming. The idea is that under sufficiently small time
steps, this algorithm produces $\varepsilon$-optimal feedback controls within the framework of open-loop control. The treatment is related to several distinct strands of the stochastic control literature. 
Impulse control problems with delayed reaction have been studied, for instance, by reducing the delayed 
intervention problem to a sequence of no-delay optimal stopping problems \cite{oksendal_sulem_delayed_reaction}. 
Closely related finite-horizon impulse control problems with decision lag and execution delay have also been analyzed 
through time-discontinuous optimal stopping formulations and discrete-time approximations \cite{palczewski_stettner_delay}. 
Separately, regression Monte Carlo methods have been developed for stochastic impulse control, where continuation and 
intervention functions are learned from simulated trajectories \cite{ludkovski_regression_impulse}. 
The weak-convergence argument used below to pass from the discrete dynamic programming problem to the continuous-time 
value function follows the general philosophy of Markov-chain approximation methods for continuous-time 
stochastic control \cite{kushner_dupuis_2001}.

The convergence properties of regression-based dynamic programming methods for discrete-time stochastic control problems 
with continuous state spaces have already been studied rigorously in \cite{regression}. For this reason, the present work 
does not further analyze the statistical convergence of the regression procedure itself, but instead focuses on the 
complementary question of convergence of the underlying discretized delayed control problem to the original continuous-time value function.

The stochastic dynamics used in the numerical examples are calibrated on forecast and 
observation data obtained from OpenWeather \cite{openweather} through a 
multidimensional Ornstein--Uhlenbeck model with correlated noise structure. 
We adopt the idea of \cite{renzo} to regard the deviation of the environmental 
parameters from their forecast as a zero-mean stochastic process. The real-data 
simulations are implemented within a rolling horizon framework over sufficiently 
short decision intervals, so that the progressive deterioration of forecast quality over 
longer horizons can be neglected within each optimization window. The use of 
Ornstein--Uhlenbeck dynamics is consistent with established stochastic models for 
environmental variables such as temperature processes \cite{benth_temperature_2005}. 
More general bounded diffusion models, such as Wright--Fisher dynamics, are equally 
compatible with the theoretical framework developed in the present work and may 
be particularly natural when modeling bounded quantities such as relative humidity. 
In fact, \cite{renzo} also develops a rigorous calibration methodology for such models.
However, application of the latter to our case would unnecessarily enlarge the numerical
study.

The main contributions of this work can be summarized as follows. 
First, we formulate the delayed air-conditioning problem as a weak 
stochastic optimal control problem. Second, we develop a rigorous 
foundation for the resulting delayed stochastic control problem. 
We introduce a regression-based dynamic programming algorithm 
adapted to the delayed switching structure and establish optimality of 
the associated idealized discrete dynamic programming formulation within the 
class of open-loop controls. Most importantly, we prove convergence of the idealized discrete dynamic 
programming value functions to the value function of the original 
continuous-time stochastic control problem. Finally, we combine the 
method with a calibrated forecast-tracking model and demonstrate 
its practical performance in a rolling-horizon study based on real 
environmental data, providing a complete pipeline from forecast 
uncertainty quantification to delayed 
stochastic control and operational cost comparison.

The structure of the paper is as follows. In Section 2 we formulate the 
delayed stochastic control problem. Section 3 develops the regression-based 
dynamic programming approximation and establishes optimality of the 
corresponding discrete dynamic programming principle.
Section 4 establishes convergence of the discrete approximation scheme. 
Finally, Section 5 presents calibration procedures and numerical examples 
for air-conditioning systems with delayed operational regimes.

\section{Weak Formulation of the Delayed SOC Problem}

This section establishes the precise formulation of the problem,
but also introduces the assumptions on the herein defined objects that shall
hold throughout the entire paper. We fix from now on a delay
variable \(\tau \in (0,\infty)\), which is the warm-up time of the Air
Conditioning device. Let
\[
b_i,\sigma_{ij}:[0,\infty)\times\mathbb R^d\to\mathbb R
\]
be Borel-measurable functions. We define the drift vector
\[
b(t,x)
=
(b_1(t,x),\ldots,b_d(t,x)),
\]
and the dispersion matrix
\[
\sigma(t,x)
=
(\sigma_{ij}(t,x))_{1\leq i\leq d,\ 1\leq j\leq r}.
\]
Throughout the paper we assume that there exists a constant \(L>0\) such that
\[
|b(t,x)-b(t,y)|
+
\|\sigma(t,x)-\sigma(t,y)\|
\leq
L|x-y|,
\]
and
\[
|b(t,x)|
+
\|\sigma(t,x)\|
\leq
L(1+|x|),
\]
for all \(t\in[0,T]\) and \(x,y\in\mathbb R^d\). In particular, the
corresponding stochastic differential equation admits a unique strong
solution.

The control problem is formulated in weak form. Fix
\[
(t,x,z)\in[0,T]\times\mathbb R^d\times[0,\infty).
\]
We minimize over systems of processes
\[
\Xi
=
(W,X,\alpha,\zeta)
\]
satisfying the following conditions:

\begin{enumerate}

\item
\(W\) is an \(r\)-dimensional $\mathcal{F}$-Brownian motion, where
$\mathcal{F}$ is the filtration generated by the full set of processes
\((W,X,\alpha,\zeta)\).

\item
\(X\) is a solution of
\begin{align*}
X_s
=
x
+
\int_t^s b(r,X_r)\,dr
+
\int_t^s \sigma(r,X_r)\,dW_r,
\qquad
t\leq s\leq T.
\end{align*}

\item
The control process
\[
\alpha:[t,T]\to\{0,1\}
\]
is left continuous with right limits.

\item
For the warm-up clock $\zeta$, we impose the initial compatibility condition
\[
\alpha_t = \mathbf{1}_{(0,\infty)}(z),
\qquad
\zeta_t = z.
\]
We write
\[
\rho_s
=
\sup\left(
\{t-z\}
\cup
\{r\in[t,s]:\alpha_r=0\}
\right),
\]
and define the clock process by
\[
\zeta_s
=
\alpha_s\bigl(s-\rho_s\bigr),
\qquad s\in[t,T].
\]
The clock therefore measures the elapsed uninterrupted time in the On regime,
initialized at $z$. The variable $\rho_s$ determines whether the system has
been switched Off since the initial time and, accordingly, whether the clock
continues from its initial value $z$ or restarts from zero. An immediate
intervention at $t$ is represented by the right limit $\alpha_{t+}$ and
incurs the switching cost $c|\alpha_{t+}-\alpha_t|$.

\end{enumerate}
We denote the collection of systems satisfying these conditions by
\[
\Lambda_{t,x,z}.
\]
Moreover, \(\alpha_s\) takes values in \(\{0,1\}\) and is assumed to be left
continuous with right limits.
Let
\[
h:
[0,T]\times\mathbb R^d\times\mathbb R\times\{0,1\}
\to
\mathbb R
\]
be an indicator cost of the type
\begin{align*}
h(t,x,z,\alpha)
=
&\;
C_1\alpha
+
C_2(1-\alpha)\mathbbm{1}_{\mathcal F}(x)
+
C_2\alpha\mathbbm{1}_{\mathcal F}(x)
\mathbbm{1}_{[0,\tau)}(z),
\end{align*}
where \(\mathcal F\) is an open set with Lipschitz boundary. We assume
throughout this work that the process \(X\) has zero occupation time at the
boundary of \(\mathcal F\).
For \(\Xi\in\Lambda_{t,x,z}\), we define
\begin{align}
J(t,x,z,\Xi)
=
\mathbb E^{\Xi}
\bigg[
\int_t^T
h(s,X_s,\zeta_s,\alpha_s)\,ds
+
c
\sum_{s\in[t,T)}
|\alpha_s-\alpha_{s+}|
\bigg].
\end{align}
Although the running cost depends directly on the control \(\alpha\), adding it
as a state variable is redundant. By left continuity, it can directly be
inferred from the clock as
\[
\alpha_t
=
\mathbbm{1}_{(0,\infty)}(\zeta_t).
\]
The value function is given by
\begin{align}
V(t,x,z)
=
\inf_{\Xi\in\Lambda_{t,x,z}}
J(t,x,z,\Xi).
\end{align}

\section{Regression-Based Discrete Dynamic Programming Algorithm}

\subsection{Idealized discrete-time control problem}

In order to construct a numerical approximation of the continuous-time
switching problem, we first introduce an idealized discrete-time stochastic
control problem. The regression algorithm presented afterwards should be
understood as a numerical approximation of the corresponding discrete dynamic
programming recursion.
Let
\[
D
=
\{0=t_0<t_1<\dots<t_K=T\}
\]
be the control discretization grid. We consider the subset
\[
\mathcal A^D
\]
of admissible controls which are constant between consecutive points of
\(D\), so that control decisions may only be taken at the times
\(t_0,\ldots,t_{K-1}\). More precisely, for
\(\alpha\in\mathcal A^D\),
\[
\alpha_s
=
\alpha_{t_{k + 1}},
\qquad
s\in(t_k,t_{k+1}],
\qquad
0\leq k\leq K-1.
\]
The state process \(X\) is not discretized and is given by the solution of the
original state equation.
For \(\alpha\in\mathcal A^D\), we define the corresponding discrete-control
cost by
\begin{align*}
J(t_k,x,z;\alpha)
=
\mathbb E
\Bigg[
&
\int_{t_k}^T
h(s,X_s,\zeta_s,\alpha_s)\,ds
\\
&\quad
+
c\sum_{\ell=k}^{K-1}
\left|
\alpha_{t_{\ell + 1}}
-
\mathbbm{1}_{(0,\infty)}(\zeta_{t_\ell})
\right|
\;\Bigg|\;
(X_{t_k},\zeta_{t_k})=(x,z)
\Bigg].
\end{align*}
The corresponding value function is given by
\[
V_{t_k}(x,z)
=
\inf_{\alpha\in\mathcal A^D}
J(t_k,x,z;\alpha).
\]
The terminal condition is prescribed by
\[
V_{t_K}(x,z)=0.
\]
We now derive the corresponding dynamic programming recursion. Assume that the
value function
\[
V_{t_{k+1}}
\]
has already been determined. For every immediate control decision
\[
\hat\alpha\in\{0,1\},
\]
the associated continuation function is defined by
\begin{align*}
C_{t_k}^{\hat\alpha}(x,z)
=
\mathbb E
\Bigg[
&
\int_{t_k}^{t_{k+1}}
h
\bigl(
s,
X_s,
\hat\alpha(z+s-t_k),
\hat\alpha
\bigr)\,ds
+
c
\left|
\hat\alpha
-
\mathbbm{1}_{(0,\infty)}(z)
\right|
\\
&\qquad
+
V_{t_{k+1}}
\left(
X_{t_{k+1}},
(z+t_{k+1}-t_k)\hat\alpha
\right)
\;\Bigg|\;
X_{t_k}=x
\Bigg].
\end{align*}
Here the clock evolution over the interval
\((t_k,t_{k+1}]\) is deterministic once the immediate control decision
\(\hat\alpha\) and the initial clock value \(z\) are fixed.

The discrete-time dynamic programming principle is therefore given recursively
by
\begin{align*}
V_{t_k}(x,z)
=
\min_{\alpha\in\{0,1\}}
C_{t_k}^{\alpha}(x,z),
\end{align*}
with associated optimal feedback control
\begin{align*}
\alpha_{t_k}^*(x,z)
=
\argmin_{\alpha\in\{0,1\}}
C_{t_k}^{\alpha}(x,z).
\end{align*}
In principle, this recursion completely characterizes the discretized optimal
control problem. However, exact evaluation of the conditional expectations
appearing in the continuation operators is infeasible. To overcome this
difficulty, we build upon the regression-based approximate dynamic programming
approach introduced by Longstaff and Schwartz in the seminal work
\cite{longstaff}. More precisely, we use Monte Carlo simulation and regression
to approximate the conditional continuation values associated with the
stochastic state process \(X\). The clock variable \(\zeta\), on the other hand,
evolves deterministically once the control is fixed. Its attainable values are
therefore induced directly by the control grid and can be tracked explicitly
during the backward recursion. In particular, this avoids the need to introduce
randomized controls in the numerical procedure.
The complete procedure is summarized in \Cref{algo}. We note that the clock
process can be capped at \(\tau\) without any effect. We make use of this to
limit the number of regressions in the numerical procedure in \Cref{algo}, but
leave the clock process uncapped in the remaining theoretical discussion for
ease of notation. The resulting dynamic programming recursion can be solved in
a single backward pass and avoids the need for policy iteration procedures such
as Howard's algorithm \cite{howard}. In the implementation considered below,
the continuation functions are approximated through polynomial regression,
where the polynomial basis is chosen sufficiently rich to capture the geometry
of the decision boundaries while maintaining numerical stability of the
backward recursion. We emphasize that polynomials are merely a choice of
convenience.
\begin{algorithm}
\caption{Longstaff--Schwartz regression algorithm}
\label{algo}
\begin{enumerate}

\item Choose a control grid
\[
D
=
\{0=t_0<\dots<t_K=T\}.
\]
The grid of clock states is induced by \(D\): at each time \(t_k\), only
clock values that can be reached by the deterministic clock dynamics along
the control grid are considered. Denote the corresponding set of attainable
clock values by \(D_z^k\).

Choose a finite-dimensional approximation space
\[
\Phi=\operatorname{span}\{\psi_1,\dots,\psi_M\}
\]
of functions on the stochastic state space.

\item Generate \(N\) Monte Carlo trajectories of the uncontrolled stochastic
state process (using a scheme of choice),
\[
X^1,\dots,X^N.
\]

\item Initialize the pathwise terminal costs by
\[
Y_K^n(z)=0,
\qquad
n=1,\dots,N,
\quad z\in D_z^K.
\]

\item For
\[
k=K-1,\dots,0,
\]
perform backward iteration. For each immediate control
\[
\hat\alpha\in\{0,1\},
\]
each simulated trajectory \(X^n\), and every attainable clock value
\[
z\in D_z^k,
\]
compute a sample approximation
\[
G_k^{n,\hat\alpha}(z)
\]
of
\[
\int_{t_k}^{t_{k+1}}
h
\bigl(
s,
X_s^n,
\hat\alpha(z+s-t_k),
\hat\alpha
\bigr)\,ds
+
c
\left|
\hat\alpha
-
\mathbbm{1}_{(0,\infty)}(z)
\right|,
\]
where the integral is evaluated by a suitable quadrature rule along the
simulated trajectory.

For each \(z\in D_z^k\) and
\(\hat\alpha\in\{0,1\}\), define the regression function by
\[
\phi_k^{z,\hat\alpha}
\in
\argmin_{\phi\in\Phi}
\sum_{n=1}^N
\left(
G_k^{n,\hat\alpha}(z)
+
Y_{k+1}^n
\left(
\hat\alpha(z+t_{k+1}-t_k)
\right)
-
\phi(X_{t_k}^n)
\right)^2.
\]

For each trajectory \(n\) and clock value \(z\in D_z^k\), choose
\[
\alpha_k^n(z)
\in
\arg\min_{\hat\alpha\in\{0,1\}}
\phi_k^{z,\hat\alpha}(X_{t_k}^n),
\]
and update the pathwise cost-to-go according to
\[
Y_k^n(z)
=
G_k^{n,\alpha_k^n(z)}(z)
+
Y_{k+1}^n
\left(
\alpha_k^n(z)(z+t_{k+1}-t_k)
\right).
\]

The functions
\[
C_k^{\hat\alpha}(x,z)
\approx
\phi_k^{z,\hat\alpha}(x)
\]
provide the corresponding regression approximations of the conditional
cost-to-go functions.
\end{enumerate}
\end{algorithm}

\subsection{Optimality of discrete Dynamic Programming in Open-Loop}

The convergence analysis uses the finite-horizon control model of
Bertsekas and Shreve \cite{bertsekas}. We therefore formulate the
discrete-control problem introduced above in their distributional framework.
The finite horizon is given by the control grid
\[
D
=
\{0=t_0<t_1<\cdots<t_K=T\}.
\]
The state and control spaces are
\[
S
=
\mathbb R^d\times[0,\infty),
\qquad
C
=
\{0,1\},
\]
respectively.
For each \(t_k\in D\setminus\{T\}\), let
\[
\mathcal W_{t_k}
=
C([t_k,t_{k+1}];\mathbb R^r),
\]
and let \(p_{t_k}\) denote the Wiener measure corresponding to a Brownian
motion starting from zero at \(t_k\).
By the assumptions on the coefficients, the state equation admits a unique
strong solution. Moreover, by the measurable functional representation for
stochastic differential equations established in \cite{kallenberg_sde}, there
exists a Borel-measurable map
\[
\Phi_{t_k}
:
\mathbb R^d\times\mathcal W_{t_k}
\longrightarrow
C([t_k,t_{k+1}];\mathbb R^d)
\]
which maps an initial state and a Brownian path segment to the corresponding
solution of the SDE on \([t_k,t_{k+1}]\).
For the state transition, only the terminal value of this local solution is
retained. We therefore define the Borel-measurable transition map
\[
f_{t_k}
:
S\times C\times\mathcal W_{t_k}
\longrightarrow
S
\]
by
\[
f_{t_k}(x,z,\alpha,w)
=
\left(
\Phi_{t_k}(x,w)(t_{k+1}),
\,
\alpha(z+\Delta t_k)
\right).
\]
The associated transition kernel is
\[
\eta^{t_k}
:
S\times C
\longrightarrow
\mathcal P(S),
\qquad
\eta^{t_k}(x,z,\alpha)
=
p_{t_k}
\circ
f_{t_k}(x,z,\alpha,\cdot)^{-1}.
\]
Since control decisions are taken precisely at the nodes of \(D\), there is
no additional state-dependent control constraint and we set
\[
U_{t_k}(x,z)
=
C
=
\{0,1\}.
\]
The running cost accumulated over a control interval depends on the entire
state trajectory on that interval. To obtain a one-stage cost depending only
on the current state and control, as required in the finite-horizon model of
Bertsekas and Shreve, we take its conditional expectation given the current
state and immediate control. Thus,
\begin{align*}
g_{t_k}(x,z,\alpha)
=
\mathbb E
\Bigg[
&
\int_{t_k}^{t_{k+1}}
h
\bigl(
s,
X_s,
\alpha(z+s-t_k),
\alpha
\bigr)\,ds
+
c
\left|
\alpha-\mathbbm{1}_{(0,\infty)}(z)
\right|
\;\Bigg|\;
X_{t_k}=x
\Bigg].
\end{align*}
Using the measurable solution map above, this can equivalently be written as
\begin{align*}
g_{t_k}(x,z,\alpha)
=
\int_{\mathcal W_{t_k}}
\Bigg[
&
\int_{t_k}^{t_{k+1}}
h
\bigl(
s,
\Phi_{t_k}(x,w)(s),
\alpha(z+s-t_k),
\alpha
\bigr)\,ds
\\
&\qquad
+
c
\left|
\alpha-\mathbbm{1}_{(0,\infty)}(z)
\right|
\Bigg]
p_{t_k}(dw).
\end{align*}
Hence, the full local solution path produced by \(\Phi_{t_k}\) is used to
evaluate the expected one-stage cost, whereas only its terminal value enters
the state transition.
The terminal cost is identically zero. Together with the prescribed initial
state
\[
(x_0,z_0)\in S,
\]
these objects define the finite-horizon control model in the sense of
Bertsekas and Shreve.
We shall use the following convention for the composition of stochastic kernels;
see \cite{kallenberg} for this notation.. Let $U, V, T$ be Borel spaces and
\[
\mu: U \to\mathcal P(V), \quad \nu : U \times V \to\mathcal P(T)
\]
be stochastic kernels. We obtain a stochastic kernel through $\mu \otimes \nu : U \to \mathcal P(V \times T)$ through 
\[
(\mu\otimes\nu)_u f
=
\int \mu_u(dv)\int \nu_{u,v}(dt) f(v,t) ,
\]
Iterated products of kernels are understood recursively in this sense. 
Especially, if $\nu: V \to \mathcal{P}(T)$ is a probability kernel, 
we can make sense of $\mu \otimes \nu$ by extending $\nu$ trivially 
(i.e. constant in the new variable) to $U \times V$. In the same way we can interpret a 
measure as a stochastic kernel. Handling 
Markov processes forces us to make use of this frequently, and we 
will do so without further notice. We will need a second type of kernel composition $\mu \nu$
for kernels of the type $\mu: U \to \mathcal{P}(V), \; \nu: V \to \mathcal{P}(T)$, defined through
\[
(\mu \nu)_u f = \int \mu_u(dv) \int \nu_v(dt) f(t).
\]
Specifically for a stochastic process, the first type of composition gives the joint law and the second
type the law at final time.
A randomized Markov policy is a collection
\[
\pi
=
(\mu_{t_0},\mu_{t_1},\ldots,\mu_{t_{K-1}})
\]
of stochastic kernels
\[
\mu^{t_k}:S\rightarrow \mathcal{P}(C),
\qquad
k=0,\ldots,K-1.
\]
We denote the set of all such policies by
\(\Pi^D\).

For an initial state \((x_0,z_0)\in S\), the joint distribution of the controlled
state and control process is given by
\[
r_K\bigl(\pi,(x_0,z_0)\bigr)
=
\delta_{x_0, z_0}
\otimes
\mu^{t_0}
\otimes
\eta^{t_0}
\otimes
\mu^{t_1}
\otimes
\cdots
\otimes
\eta^{t_{K-2}}
\otimes
\mu^{t_{K-1}},
\]
which is a measure on
\[
S \times C \times  \hdots \times S \times C,
\]
with the appropriate amount of repetitions. The cost associated with a policy \(\pi\in\Pi^D\) is
\begin{align}
J_\pi^D(x_0,z_0)
=
\int
\sum_{k=0}^{K-1}
g_{t_k}(x_k,z_k,\alpha_{t_k})
\,
dr_K\bigl(\pi,(x_0,z_0)\bigr),
\end{align}
and the optimal cost is
\begin{align}
J^{D,*}(x_0,z_0)
=
\inf_{\pi\in\Pi^D}
J_\pi^D(x_0,z_0).
\end{align}

Next, we define the dynamic programming operators. For
\(1\leq k\leq K-1\),
\begin{align}
T_{t_k}(J)(x,z)
=
\inf_{\alpha\in C}
\left\{
g_{t_{k-1}}(x,z,\alpha)
+
\int_S
J(x',z')
\,d\eta^{t_{k-1}}\bigl((x,z),\alpha\bigr)
\right\},
\end{align}
and
\begin{align}
T_{t_K}(J)(x,z)
=
\inf_{\alpha\in C}
g_{t_{K-1}}(x,z,\alpha),
\end{align}
where we adopt the convention that \(J_{t_K}=0\).
We summarize below the finite-horizon optimality result from
\cite{bertsekas}, adapted to the setting of the present paper, which
establishes the optimality of the dynamic programming recursion over
randomized Markov policies.
\begin{proposition}\label{prop:bertsekas}
Let \(J_{t_K}\) be the zero function on \(S\). Then
\begin{align}\label{opt:dyn_prog}
J^{D,*}
=
T_{t_1}\circ\hdots\circ T_{t_K}(J_{t_K}).
\end{align}
There exists an optimal strict control, i.e. there are measurable functions
\[
u_{t_k}:S\to C,\qquad 0\leq k\leq K-1,
\]
such that
\[
J^{D,*}=J_\pi^D,
\]
where
\[
\pi=(\mu_{t_0},\hdots,\mu_{t_{K-1}})
\]
is given by
\[
\mu_{t_k}(x,z)
=
\delta_{u_{t_k}(x,z)}.
\]
Indeed, \(u_{t_k}\) is given through one of the possible measurable selections by
\begin{align}
u_{t_k}(x,z)
\in
\argmin_{u\in C}
\left\{
g_{t_k}(x,z,u)
+
\int_S
J(x',z')
\,d\eta^{t_k}((x,z),u)
\right\}.
\end{align}
\end{proposition}

\begin{proof}
The proof is due to \cite{bertsekas}. It requires a combination of 
Proposition~8.2 and a reduction of the nonstationary model as in 
Chapter 10.1. 
New is only that the dynamic programming algorithm gives a proper 
measurable feedback control, which is a direct consequence of the 
finiteness of the control space. A possible measurable selection is
\begin{align*}
u_{t_k}(x,z) = \mathbbm{1}_C(x,z),
\end{align*}
where
\begin{align*}
C = \Bigg\{
(x,z)\;\Bigg|\;
g_{t_k}(x,z,1)
& +
\int_S
J_{t_k}(x',z')
\,d\eta^{t_k}((x,z),1) \\
& \leq
g_{t_k}(x,z,0)
+
\int_S
J_{t_k}(x',z')
\,d\eta^{t_k}((x,z),0)
\Bigg\}.
\end{align*}
\end{proof}
If we are given the uncontrolled state process \(X\) and an initial condition
\(z_0\) for the clock, we can use the feedback control functions \(u_{t_k}\)
from \Cref{prop:bertsekas} to define recursively a feedback policy on the grid
\(D\) through
\begin{align*}
& \zeta^f_0 = z_0, \\
& \alpha^f_{t_k}
=
u_{t_k}(X^f_{t_k},\zeta^f_{t_k}), \\
& \zeta_{t_{k+1}}^f
=
\alpha_{t_k}^f
\bigl(
\zeta_{t_k}^f
+
\Delta t_k
\bigr).
\end{align*}
We extend the grid control to a process on \([0,T]\) by keeping it constant
between consecutive control dates. That is,
\[
\alpha_s^f
=
\alpha_{t_k}^f,
\qquad
s\in(t_k,t_{k+1}],
\]
so that
\[
\alpha^f\in\mathcal A^D.
\]
We refer to \(\alpha^f\) as the continuation of the discrete feedback policy.
It is a feedback control only at the grid points and should not be interpreted
as a continuous-time feedback control.
The next proposition is essential when comparing approximations of
continuous-time controls to continuations of feedback controls on the grid.
The former are generally not of feedback type.

\begin{proposition}\label{feedback}
Let \((W,X,\zeta,\alpha)\) be an admissible system of the
continuous-time problem with \(\alpha\in\mathcal A^D\). Then we can choose an
optimal feedback control on the grid and construct \(\alpha^f\) as above, such that we have for the
corresponding state process \((X^f,\zeta^f)\)
\[
J(0,x_0,z_0;\alpha^f)
\leq
J(0,x_0,z_0;\alpha).
\]
\end{proposition}

\begin{proof} We consider this result classical. Nevertheless, for reference, we provide a proof in \Cref{app:proofprop}.
\end{proof}

\section{Convergence of the Discrete Algorithm}
We next prove convergence of the discrete dynamic programming approximation to
the continuous-time value function. The upper bound is obtained by discretizing
an optimal continuous-time control. The lower bound is obtained by extracting a
weak limit from a sequence of discrete optimizers and identifying the limit as
an admissible continuous-time control.
Let
\[
D^m
=
\{0=t_0^m<t_1^m<\cdots<t_{K_m}^m=T\}
\]
be a sequence of control grids with mesh size tending to zero as
\(m\to\infty\). For each \(m\), we restrict the admissible controls to
\(\mathcal A^{D^m}\), i.e. controls which are constant between consecutive
points of \(D^m\) and may change only at the control dates.

The state process \(X\) is not discretized and is given by the solution of the
original continuous-time state equation. For a control
\(\alpha^m\in\mathcal A^{D^m}\), the corresponding clock process
\(\zeta^m\) is defined as in the continuous-time problem. In particular, at
the control dates,
\[
\zeta^m_{t_{k+1}^m}
=
\alpha^m_{t_k^m}
\left(
\zeta^m_{t_k^m}
+
\Delta t_k^m
\right).
\]
Thus, the discrete dynamic programming problem differs from the original
continuous-time control problem only through the restriction of the admissible
control class to \(\mathcal A^{D^m}\).
\begin{theorem}\label{thm:discrete_convergence}
Let
\[
\Xi^{m,f}
=
(W^{m,f},X^{m,f},\alpha^{m,f},\zeta^{m,f})
\]
be an optimal feedback-controlled system obtained from the idealized dynamic programming
recursion on the grid \(D^m\), with
\[
\alpha^{m,f}\in\mathcal A^{D^m}.
\]
Then
\[
J(0,x_0,z_0;\Xi^{m,f})
\longrightarrow
V(0,x_0,z_0)
\]
as
\[
|D^m|
\underset{m\to\infty}{\longrightarrow}
0.
\]
\end{theorem}

\begin{proof}

We first prove the upper bound
\[
\limsup_{m\to\infty}
J(0,x_0,z_0;\Xi^{m,f})
\leq
V(0,x_0,z_0).
\]
Let $\varepsilon>0$. By the definition of the value function, there exists
an admissible system
$
\Xi^\varepsilon
=
(W^\varepsilon,X^\varepsilon,\alpha^\varepsilon,\zeta^\varepsilon)
\in\Lambda_{0,x_0,z_0}
$
such that
\[
J(0,x_0,z_0;\Xi^\varepsilon)
\leq
V(0,x_0,z_0)+\varepsilon.
\]
For every $m$, define a control
\[
\hat\alpha^{m,\varepsilon}\in\mathcal A^{D^m}
\]
by projecting $\alpha^\varepsilon$ onto the grid $D^m$ and keeping it
constant between consecutive control dates. Let
$\hat\zeta^{m,\varepsilon}$ be the corresponding clock process and set
\[
\hat\Xi^{m,\varepsilon}
=
(W^\varepsilon,X^\varepsilon,
\hat\alpha^{m,\varepsilon},\hat\zeta^{m,\varepsilon}).
\]
The strictly positive switching penalty $c>0$ provides a bound on the
expected number of switches for any admissible control with finite cost,
since the running cost is bounded from below. In particular,
$\alpha^\varepsilon$ has almost surely finitely many switches.
Consequently, on almost every sample path, the distinct switching times
are separated by a strictly positive distance. For sufficiently fine
grids, the projection therefore preserves the number and order of
switches, while the corresponding switching times converge to those of
$\alpha^\varepsilon$. Hence, the total switching penalty of the
projected control eventually coincides with that of the original
control. Moreover, since the clock is determined by the initial clock
state and the most recent shutdown, convergence of the switching times
implies convergence of the corresponding clock processes at every
continuity point. Thus,
\[
(\hat\zeta^{m,\varepsilon},\hat\alpha^{m,\varepsilon})
\longrightarrow
(\zeta^\varepsilon,\alpha^\varepsilon)
\]
almost everywhere on $[0,T]$, almost surely. Since $h$ is bounded
and its discontinuities are contained in a set of zero occupation
measure under the state process, dominated convergence yields
convergence of the expected running costs. Together with preservation
of the switching penalty, this gives
\[
J(0,x_0,z_0;\hat\Xi^{m,\varepsilon})
\longrightarrow
J(0,x_0,z_0;\Xi^\varepsilon),
\qquad m\to\infty.
\]

By \Cref{feedback}, the continuation of the discrete feedback policy is
optimal among controls in $\mathcal A^{D^m}$. Therefore,
\[
J(0,x_0,z_0;\Xi^{m,f})
\leq
J(0,x_0,z_0;\hat\Xi^{m,\varepsilon}).
\]
Taking the limit superior yields
\[
\limsup_{m\to\infty}
J(0,x_0,z_0;\Xi^{m,f})
\leq
J(0,x_0,z_0;\Xi^\varepsilon)
\leq
V(0,x_0,z_0)+\varepsilon.
\]
Since $\varepsilon>0$ was arbitrary, we conclude that
\[
\limsup_{m\to\infty}
J(0,x_0,z_0;\Xi^{m,f})
\leq
V(0,x_0,z_0).
\]

For the lower bound, observe that for every \(m\),
\[
\Xi^{m,f}
=
(W^{m,f},X^{m,f},\alpha^{m,f},\zeta^{m,f})
\in\Lambda_{0,x_0,z_0},
\]
with
\[
\alpha^{m,f}\in\mathcal A^{D^m}.
\]
Hence, by the definition of the value function,
\[
V(0,x_0,z_0)
\leq
J(0,x_0,z_0;\Xi^{m,f})
\]
for every \(m\). Consequently,
\[
V(0,x_0,z_0)
\leq
\liminf_{m\to\infty}
J(0,x_0,z_0;\Xi^{m,f}).
\]

Combining the upper and lower bounds gives
\[
J(0,x_0,z_0;\Xi^{m,f})
\longrightarrow
V(0,x_0,z_0),
\qquad
m\to\infty.
\]
\end{proof}

\section{Numerical Examples}

The numerical experiments are intended to illustrate the complete data-driven workflow. We first describe the available forecast and observation data and the forecast-tracking stochastic model. We then discuss the monthly calibration procedure. Afterwards, we investigate the qualitative structure of the computed feedback policies by means of policy heat maps and representative trajectories. Finally, we evaluate the method in a rolling-horizon implementation on Berlin weather data and compare the resulting costs against a permanently active climate-control strategy.

\subsection{Weather Forecast and Observation Data}

The data consist of weather forecasts and historical realized weather observations obtained from OpenWeather \cite{openweather}. Forecasts are issued every six hours and each forecast provides predictions for the next sixteen days. We always use the most recent forecast in a non-overlapping way. In the experiments below we use temperature and humidity as the environmental variables. The forecast trajectory is denoted by
\[
\mu(t)=\bigl(\mu_1(t),\ldots,\mu_d(t)\bigr).
\]
The purpose of the stochastic model is to construct sample paths which track this moving forecast, reproduce realistic forecast errors, and can be used as input trajectories for the stochastic control algorithm.

\subsection{Forecast-Tracking Environmental Dynamics}

We model temperature and relative humidity jointly by the two-dimensional
forecast-tracking stochastic differential equation
\begin{align*}
dX_t^1
&=
\Bigl(
\theta_1(\mu_1(t)-X_t^1)+\dot\mu_1(t)
\Bigr)\,dt
+
\sigma_{11}\,dW_t^1,
\\
dX_t^2
&=
\Bigl(
\theta_2(\mu_2(t)-X_t^2)+\dot\mu_2(t)
\Bigr)\,dt
+
\sigma_{21}\,dW_t^1
+
\sigma_{22}\,dW_t^2,
\end{align*}
where \(W=(W^1,W^2)\) is a two-dimensional Brownian motion,
\(\theta_1,\theta_2>0\), and
\[
\sigma
=
\begin{pmatrix}
\sigma_{11} & 0\\
\sigma_{21} & \sigma_{22}
\end{pmatrix}.
\]
The deterministic trajectory
\[
\mu(t)
=
\bigl(\mu_1(t),\mu_2(t)\bigr)
\]
denotes the corresponding temperature and humidity forecast. This structure
follows the forecast-tracking idea used in \cite{renzo}. Introducing the centered process
\[
Y_t
=
X_t-\mu(t)
\]
removes the inhomogeneous drift and gives
\begin{align*}
dY_t^1
&=
-\theta_1Y_t^1\,dt
+
\sigma_{11}\,dW_t^1,
\\
dY_t^2
&=
-\theta_2Y_t^2\,dt
+
\sigma_{21}\,dW_t^1
+
\sigma_{22}\,dW_t^2.
\end{align*}
Thus, forecast deviations are modeled by a centered bivariate
Ornstein--Uhlenbeck process with lower-triangular volatility structure.
For observation times
\[
0=t_0<t_1<\cdots<t_N,
\qquad
\Delta t_n=t_{n+1}-t_n,
\]
the exact transition law is Gaussian. Writing
\[
Y_n=Y_{t_n},
\]
we have
\[
Y_{n+1}
=
A_nY_n+\varepsilon_n,
\qquad
A_n
=
\begin{pmatrix}
e^{-\theta_1\Delta t_n} & 0\\
0 & e^{-\theta_2\Delta t_n}
\end{pmatrix},
\]
where
\[
\varepsilon_n
\sim
\mathcal N(0,\Sigma_n)
\]
and
\begin{align*}
\Sigma_n^{11}
&=
\sigma_{11}^2
\frac{1-e^{-2\theta_1\Delta t_n}}{2\theta_1},
\\
\Sigma_n^{12}
=
\Sigma_n^{21}
&=
\sigma_{11}\sigma_{21}
\frac{1-e^{-(\theta_1+\theta_2)\Delta t_n}}
{\theta_1+\theta_2},
\\
\Sigma_n^{22}
&=
\bigl(\sigma_{21}^2+\sigma_{22}^2\bigr)
\frac{1-e^{-2\theta_2\Delta t_n}}{2\theta_2}.
\end{align*}
Equivalently, the transition law of \(X\) is Gaussian with conditional mean
\[
\mu(t_{n+1})
+
A_n\bigl(X_{t_n}-\mu(t_n)\bigr)
\]
and covariance matrix \(\Sigma_n\). To model uncertainty already present at the beginning of each optimization
window, we assume that the forecast was exact at an earlier time
\(-\delta\), i.e.
\[
X_{-\delta}
=
\mu(-\delta),
\]
or equivalently,
\[
Y_{-\delta}=0.
\]
Propagation of the centered Ornstein--Uhlenbeck process from
\(-\delta\) to \(0\) yields
\[
X_0
\sim
\mathcal N\bigl(\mu(0),\Sigma_0(\delta)\bigr),
\]
where
\begin{align*}
\Sigma_0^{11}(\delta)
&=
\sigma_{11}^2
\frac{1-e^{-2\theta_1\delta}}{2\theta_1},
\\
\Sigma_0^{12}(\delta)
=
\Sigma_0^{21}(\delta)
&=
\sigma_{11}\sigma_{21}
\frac{1-e^{-(\theta_1+\theta_2)\delta}}
{\theta_1+\theta_2},
\\
\Sigma_0^{22}(\delta)
&=
\bigl(\sigma_{21}^2+\sigma_{22}^2\bigr)
\frac{1-e^{-2\theta_2\delta}}{2\theta_2}.
\end{align*}
The parameter \(\delta\) can therefore be interpreted as an effective forecast
age and determines the uncertainty of the initial state relative to the
forecast.

\subsection{Monthly Calibration Procedure}

The stochastic forecast model is calibrated month-wise. For a given target
month, only forecast and observation pairs from the same calendar month in
previous years are used. For instance, calibration for July is based on July
data from the preceding years. This preserves the seasonal structure of the
forecast errors while providing several historical trajectories for parameter
estimation.

For each available forecast trajectory, we first form the centered observations
\[
Y_{t_n}
=
X_{t_n}-\mu(t_n).
\]
Empirical quadratic variations and covariations provide an initial estimate of
the volatility matrix. In the two-dimensional setting,
\[
\widehat C
=
\begin{pmatrix}
\widehat C_{11} & \widehat C_{12}\\
\widehat C_{12} & \widehat C_{22}
\end{pmatrix}
\approx
\begin{pmatrix}
\sigma_{11}^2
&
\sigma_{11}\sigma_{21}
\\
\sigma_{11}\sigma_{21}
&
\sigma_{21}^2+\sigma_{22}^2
\end{pmatrix}.
\]
Hence an initial value of the lower-triangular volatility matrix can be obtained
successively from
\[
\sigma_{11}^2
=
\widehat C_{11},
\qquad
\sigma_{21}
=
\frac{\widehat C_{12}}{\sigma_{11}},
\qquad
\sigma_{22}^2
=
\widehat C_{22}-\sigma_{21}^2.
\]
These estimates are used only as starting values for the subsequent likelihood
optimization.
The final parameters
\[
\theta
=
(\theta_1,\theta_2),
\;
\sigma
=
(\sigma_{11},\sigma_{21},\sigma_{22}),
\;
\delta
\]
are calibrated jointly by maximum likelihood. For a centered observation
\(Y_{t_n}\), the exact transition distribution derived above is
\[
Y_{t_{n+1}}
\mid Y_{t_n}
\sim
\mathcal N
\bigl(
A_nY_{t_n},
\Sigma_n(\theta,\sigma)
\bigr).
\]
Writing
\[
R_n
=
Y_{t_{n+1}}-A_nY_{t_n},
\]
the corresponding contribution to the negative log-likelihood, up to an
additive constant independent of the parameters, is
\[
R_n^\top
\Sigma_n(\theta,\sigma)^{-1}
R_n
+
\log\det\Sigma_n(\theta,\sigma).
\]
The initial centered observation is treated using the backward consistency
condition
\[
Y_0
\sim
\mathcal N
\bigl(
0,\Sigma_0(\delta;\theta,\sigma)
\bigr).
\]
Its contribution to the negative log-likelihood is therefore
\[
Y_0^\top
\Sigma_0(\delta;\theta,\sigma)^{-1}
Y_0
+
\log\det\Sigma_0(\delta;\theta,\sigma).
\]
Summing the initial and transition contributions over all historical forecast
trajectories used for the respective month gives the objective function for the
joint maximum likelihood calibration. Thus, the parameters
\[
(\theta_1,\theta_2,\sigma_{11},\sigma_{21},\sigma_{22},\delta)
\]
are optimized simultaneously. In particular, the quadratic-covariation
estimate of \(\sigma\) is not retained as a separate estimator, but serves only
as an initialization of the joint likelihood optimization.
Once the parameters have been calibrated, a stochastic trajectory corresponding
to a new forecast \(\mu\) subsequently propagated according to the forecast-tracking dynamics above.
Specifically, the initial condition is $
X_0
=
\mu(0)+\xi,
$
with
$
\xi
\sim
\mathcal N\bigl(0,\Sigma_0(\delta)\bigr),
$

\subsection{Gaussian evaluation of the one-stage cost}
\label{gaussian-stage-cost}

The one-stage cost in the time-discretized dynamic program is
the conditional expectation of the running cost accumulated over
an entire grid interval, together with the corresponding switching
penalty. For the Ornstein--Uhlenbeck model considered in the numerical
experiments, this expectation can be reduced explicitly to Gaussian
distribution functions.

Recall that the running cost is
\[
    h(t,x,z,\alpha)
    =
    C_1\alpha
    +
    C_2(1-\alpha)\mathbbm{1}_{\mathcal F}(x)
    +
    C_2\alpha\mathbbm{1}_{\mathcal F}(x)
    \mathbbm{1}_{[0,\tau)}(z),
\]
where
\[
    \mathcal C=[a,b]\times[e,o],
    \qquad
    \mathcal F=\mathbb R^2\setminus\mathcal C.
\]
Fix a grid interval $[t_k,t_{k+1}]$, an initial state $(x,z)$,
and an action $\alpha\in\{0,1\}$. The clock evolves according to
$\zeta_{t_k+u}=\alpha(z+u)$ for $0<u\leq\Delta t_k$.
The one-stage cost is therefore
\[
\begin{aligned}
    g_k(x,z,\alpha)
    &=
    c\left|\alpha-\mathbf1_{(0,\infty)}(z)\right|
    +C_1\alpha\Delta t_k\\
    &\quad+
    C_2\int_0^{\Delta t_k}
    \left[
        (1-\alpha)
        +\alpha\mathbf1_{[0,\tau)}(\alpha(z+u))
    \right]p_u(x)\,du,
\end{aligned}
\]
where
\[
    p_u(x)
    =
    \mathbb P\left(
        X_{t_k+u}\in\mathcal F
        \mid X_{t_k}=x
    \right).
\]

Under the Ornstein--Uhlenbeck dynamics, the conditional distribution
of $X_{t_k+u}$ is bivariate Gaussian,
\[
    X_{t_k+u}\mid X_{t_k}=x
    \sim\mathcal N(\mu_u(x),\Sigma_u).
\]
Writing $\mu_{u,i}(x)$ for the components of the conditional mean,
$\sigma_{u,i}^2$ for the diagonal entries of $\Sigma_u$, and
$\rho_u$ for the conditional correlation, define
\[
\begin{aligned}
    l_1&=\frac{a-\mu_{u,1}(x)}{\sigma_{u,1}},
    &r_1&=\frac{b-\mu_{u,1}(x)}{\sigma_{u,1}},\\
    l_2&=\frac{e-\mu_{u,2}(x)}{\sigma_{u,2}},
    &r_2&=\frac{o-\mu_{u,2}(x)}{\sigma_{u,2}}.
\end{aligned}
\]
Then
\[
\begin{aligned}
    p_u(x)
    =1-&\Bigl[
        \Phi_2(r_1,r_2;\rho_u)
        -\Phi_2(l_1,r_2;\rho_u)\\
        &\quad-\Phi_2(r_1,l_2;\rho_u)
        +\Phi_2(l_1,l_2;\rho_u)
    \Bigr],
\end{aligned}
\]
where $\Phi_2(\cdot,\cdot;\rho)$ denotes the standard bivariate
normal distribution function with correlation $\rho$.

Thus, the spatial expectation of the running cost is expressed
explicitly in terms of Gaussian distribution functions. Only
a one-dimensional deterministic integral over the grid interval
remains, which is evaluated by numerical quadrature. In particular,
no simulation of intermediate state paths or additional
time-discretization of the running cost is required.

The resulting stage cost $g_k(x,z,\alpha)$ is a deterministic function
of the current augmented state and action, and is used consistently
in both the discrete-time Bellman recursion and the primal--dual
formulation of Section~\ref{primal-dual}.

\subsection{Structure of the Optimal Feedback Control}

We first inspect the computed feedback control through policy heat maps and
representative forecast-driven sample paths. Both experiments use the same
numerical setting. Importantly, the deterministic forecast trajectory used in
these experiments is taken from real forecast data, whereas the environmental
state trajectories are sampled from the stochastic model calibrated to
historical forecast errors. 
The environmental state consists of temperature and relative humidity, and the
comfort region $\mathcal C$ and respective violation region $\mathcal F$ are given by
\[
\mathcal C
=
[18,25]\times[0,60], \quad \mathcal F = \R^2 \setminus \mathcal C,
\]
corresponding to temperatures between \(18^\circ\mathrm C\) and
\(25^\circ\mathrm C\) and relative humidity between \(0\%\) and \(60\%\).
We consider a \(24\)-hour horizon with hourly control decisions and a warm-up
time of
\[
\tau=1\ \text{hour}.
\]
The switching cost is \(c=5\), while the running-cost parameters are
\[
C_1=20,
\qquad
C_2=2000.
\]
The continuation values are approximated using polynomial regression of
degree \(22\), trained on \(512000\) simulated trajectories.
These parameters are selected based on the primal--dual
diagnostics described in \Cref{primal-dual}, with the
corresponding estimates and $95\%$ confidence intervals
reported in \Cref{tab:primal-dual}.
\begin{table}[htbp]
\centering
\caption{Primal--dual diagnostics for the discrete-time
stochastic control problem. Standard errors are estimated
from independent out-of-sample trajectories; the gap
standard error accounts for pairing.}
\label{tab:primal-dual}
\begin{tabular}{lrr}
\hline
Quantity & Estimate & Standard error \\
\hline
Dual lower estimate   & 343.8239 & 1.9419 \\
Primal upper estimate & 344.3114 & 1.5524 \\
Absolute gap          & 0.4874   & 2.2083 \\
Relative gap          & 0.1416\% & --- \\
\hline
\end{tabular}
\end{table}
More precisely, the deterministic trajectory \(\mu\) is the actual temperature
and humidity forecast for June 15, 2023, interpolated onto the numerical time
grid. The stochastic model uses the parameters obtained from the monthly
calibration for June. Conditional on the observed forecast \(\mu\), sample
paths of temperature and relative humidity are generated from the calibrated
forecast-tracking Ornstein--Uhlenbeck model using its exact transition law.
Consequently, the forecast entering the experiment is not simulated: only the
deviations of the environmental state from this forecast are sampled from the
calibrated model.
The policy heat maps visualize the resulting feedback decision as a function
of temperature and relative humidity at a fixed time. Since the delayed
activation mechanism introduces the clock variable as an additional state
variable, the policy is shown for several values of the clock. This makes it
possible to inspect how the switching boundary changes with the current
activation status and the elapsed warm-up time.
\begin{figure}[ht]
\centering
\includegraphics[width=0.72\textwidth]{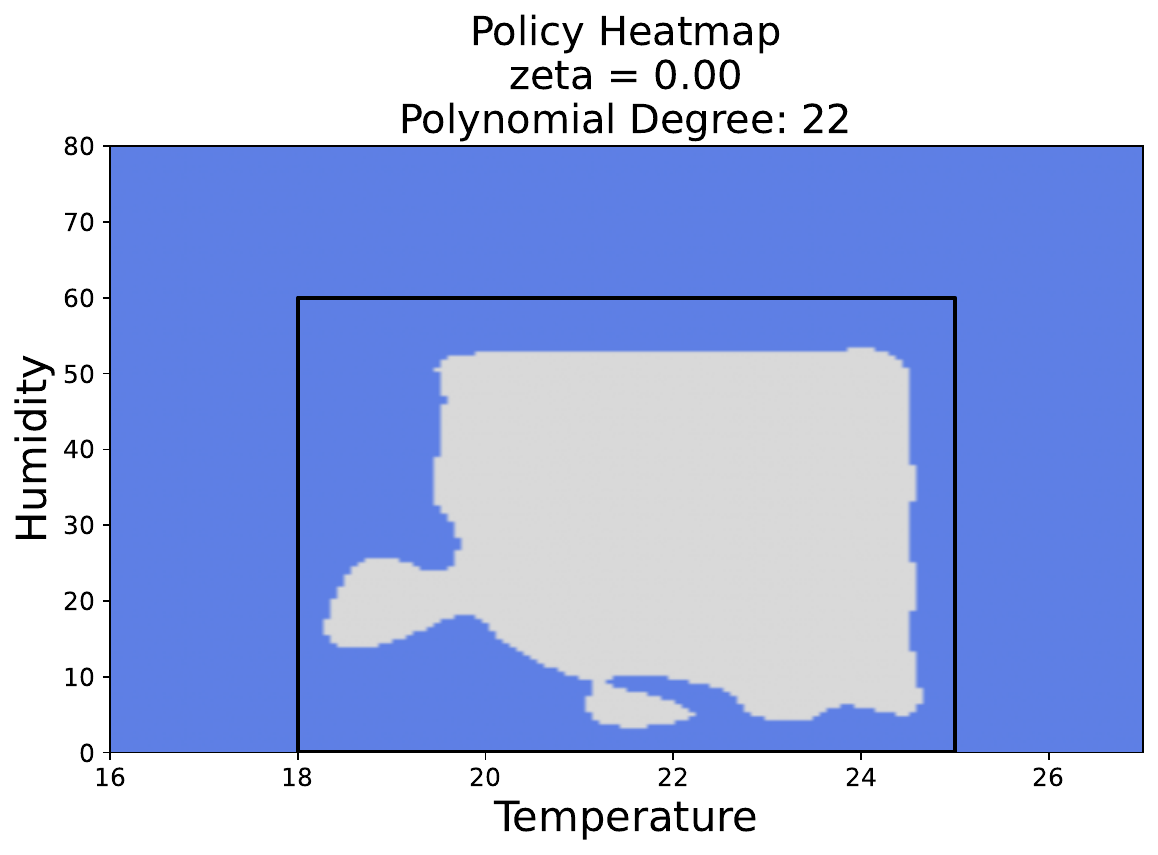}
\includegraphics[width=0.72\textwidth]{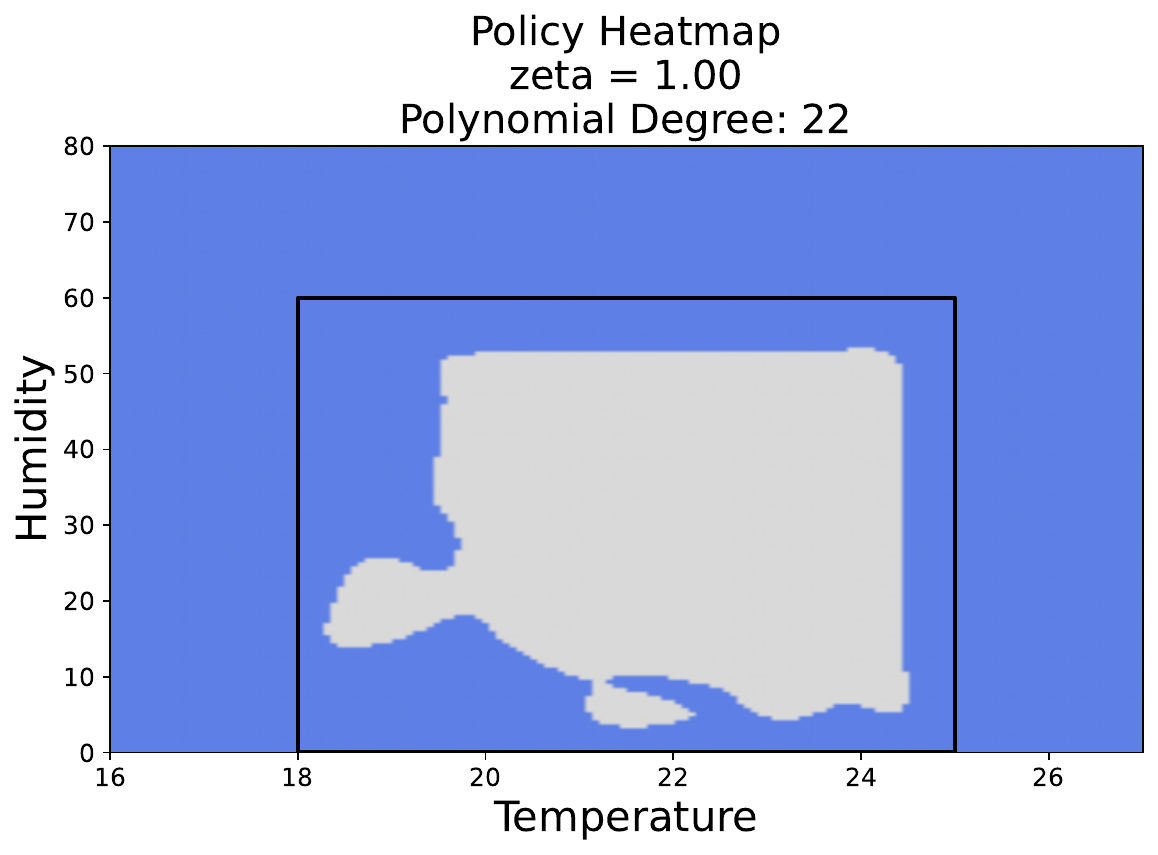}
\caption{Policy heatmap for polynomial degree $22$. 
\textbf{Black lines:} Comfort regime boundary. 
\textbf{Blue region:} Turn AC On.
\textbf{Grey region:} Turn AC Off.
\textbf{Upper plot:} 
Optimal feedback policy for $\zeta=0$.
\textbf{Lower plot:} Optimal feedback policy for $\zeta=1$.}
\label{fig:policy_heatmap_deg5}
\end{figure}

\begin{figure}[ht]
\centering
\includegraphics[width=0.72\textwidth]{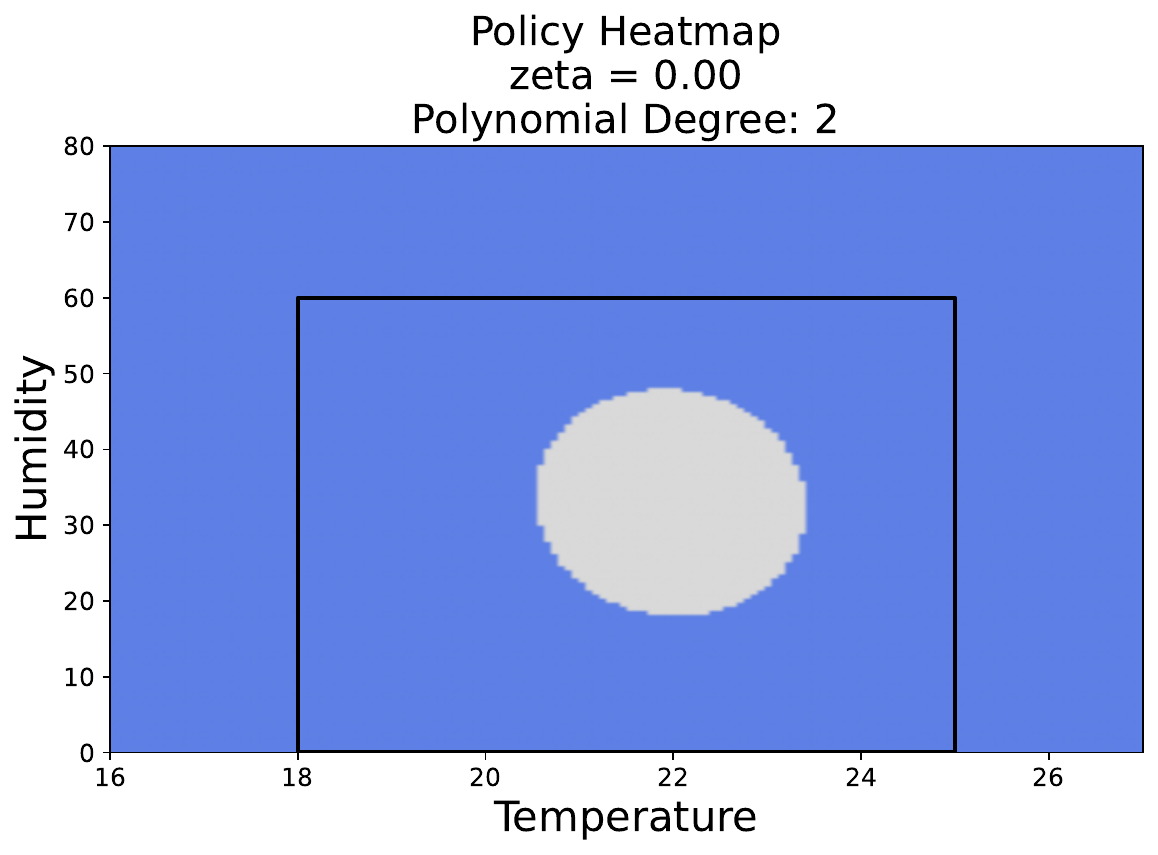}
\includegraphics[width=0.72\textwidth]{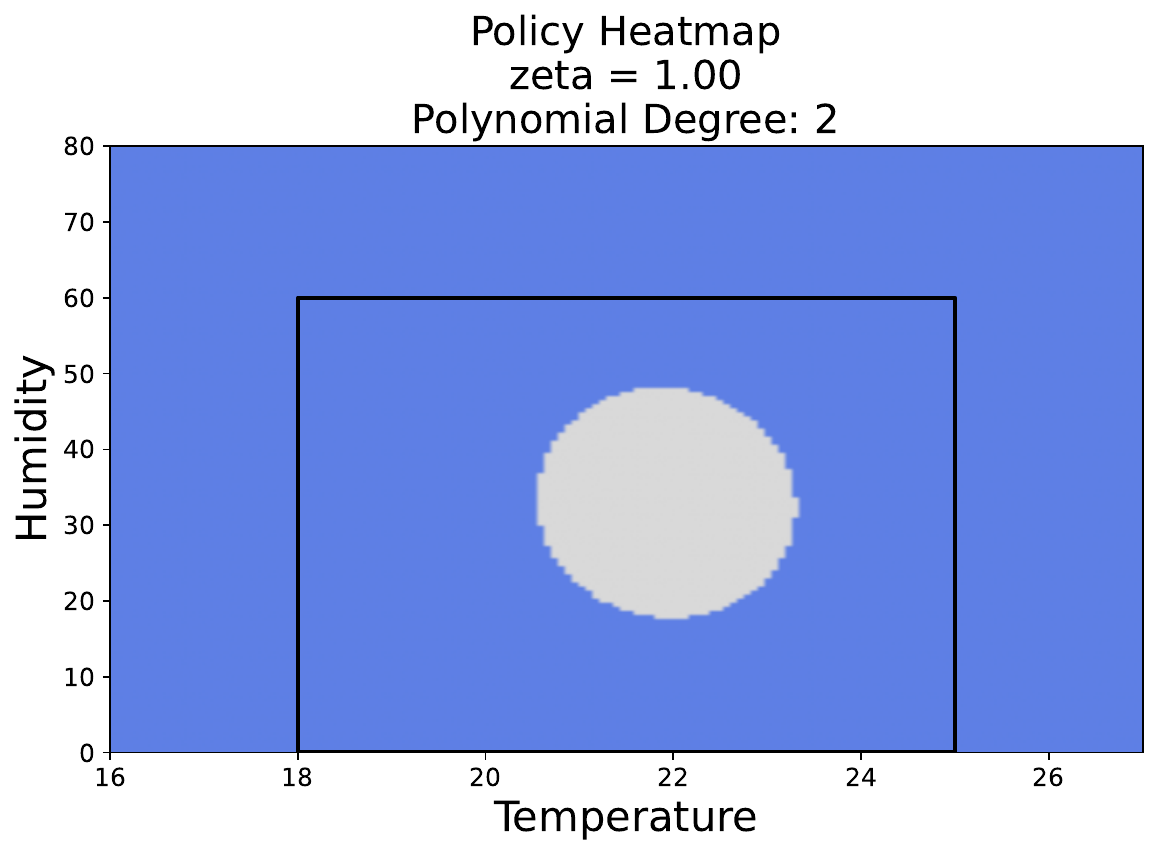}
\caption{Policy heatmap for polynomial degree $2$. 
\textbf{Black lines:} Comfort regime boundary. 
\textbf{Blue region:} Turn AC On.
\textbf{Grey region:} Turn AC Off.
\textbf{Upper plot:} 
Optimal feedback policy for $\zeta=0$.
\textbf{Lower plot:} Optimal feedback policy for $\zeta=1$.}
\label{fig:policy_heatmap_deg2}
\end{figure}

\subsection{Representative Policy Trajectories}

We next examine the temporal behavior of the computed policy along individual
realizations. For these plots, we again condition on the true temperature and
humidity forecast for June 15, 2023, and independently sample environmental
trajectories around this forecast from the calibrated June model. The plotted
temperature and humidity paths are therefore synthetic samples from the
calibrated conditional model, while the forecast trajectories shown alongside
them are the corresponding real forecast data.

The test trajectories are generated independently of the \(8000\) trajectories
used to compute the policy. For each representative realization, the figures
show the sampled temperature and humidity trajectories together with the true
forecast and the corresponding comfort intervals. The final panel displays the
binary control process, with value one indicating that the air-conditioning
system is switched on. The plots thereby illustrate how the computed policy
responds to possible realizations of the environmental conditions around an
actually observed forecast, and in particular how forecast uncertainty,
comfort violations, the one-hour activation delay, and switching costs interact
in the resulting control decisions. The comparison in \Cref{fig:policy_heatmap_deg2,fig:policy_heatmap_deg5} shows the qualitative effect of the delay. When $\zeta=0$, activation must occur before the boundary of the comfort region is reached. Otherwise the system would still be in the warm-up phase when the violation occurs. Once $\zeta=1$, the controller reacts more carefully, since the system is already in regime and turning off carries a larger risk. This anticipatory behaviour is one of the main qualitative consequences of the delayed switching structure.

\begin{figure}[ht]
\centering
\includegraphics[width=0.8\textwidth]{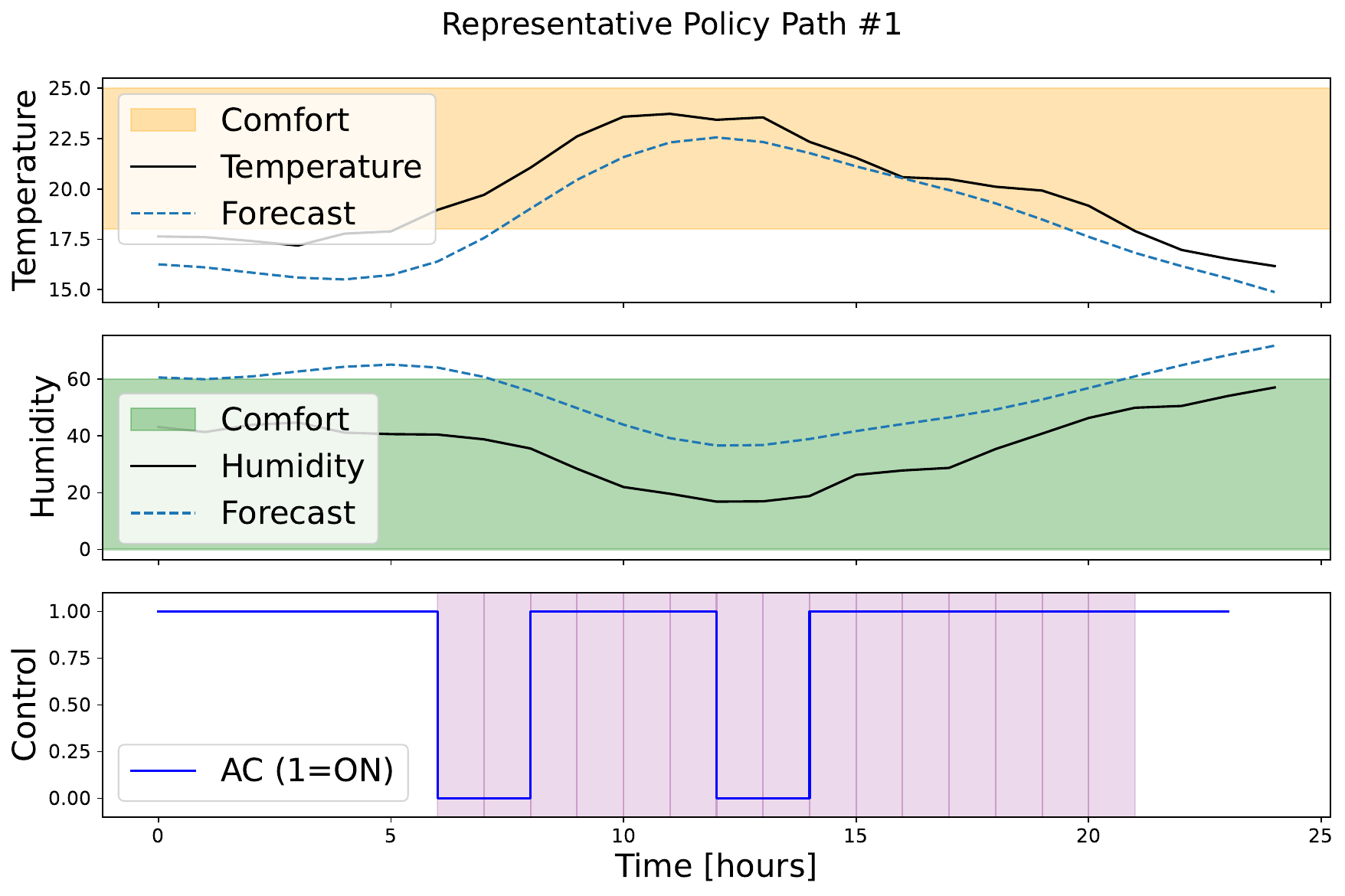}
\caption{Representative policy rollout under synthetic data and real forecasts. The control remains inactive while the environmental variables stay safely inside the comfort region and activates before the violation becomes critical.}
\label{fig:policy_path_1}
\end{figure}

\begin{figure}[ht]
\centering
\includegraphics[width=0.8\textwidth]{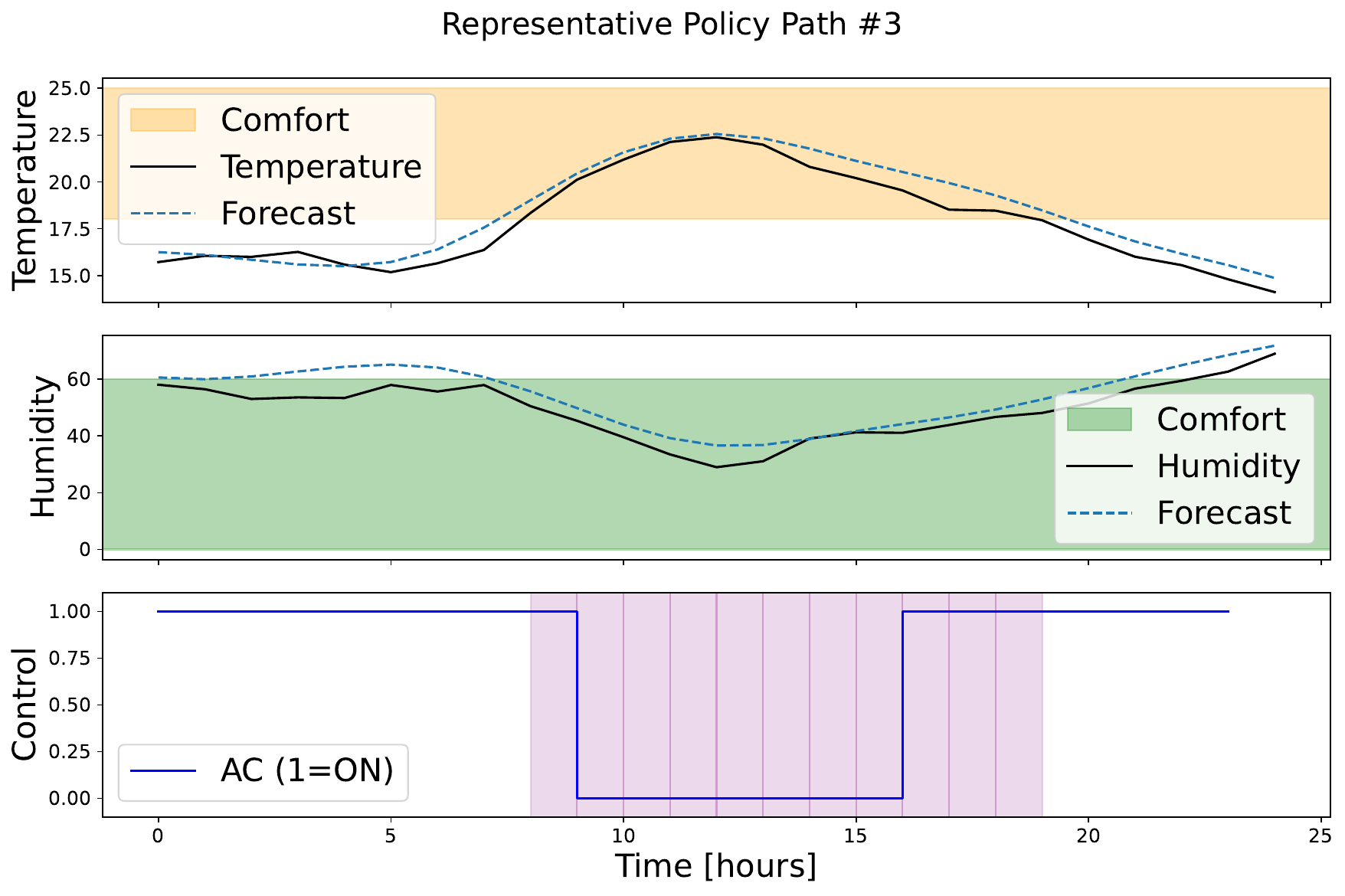}
\caption{Representative policy rollout under synthetic data and real forecasts. The switching times reflect the interaction between forecast uncertainty and the delayed warm-up period.}
\label{fig:policy_path_3}
\end{figure}

\begin{figure}[ht]
\centering
\includegraphics[width=0.8\textwidth]{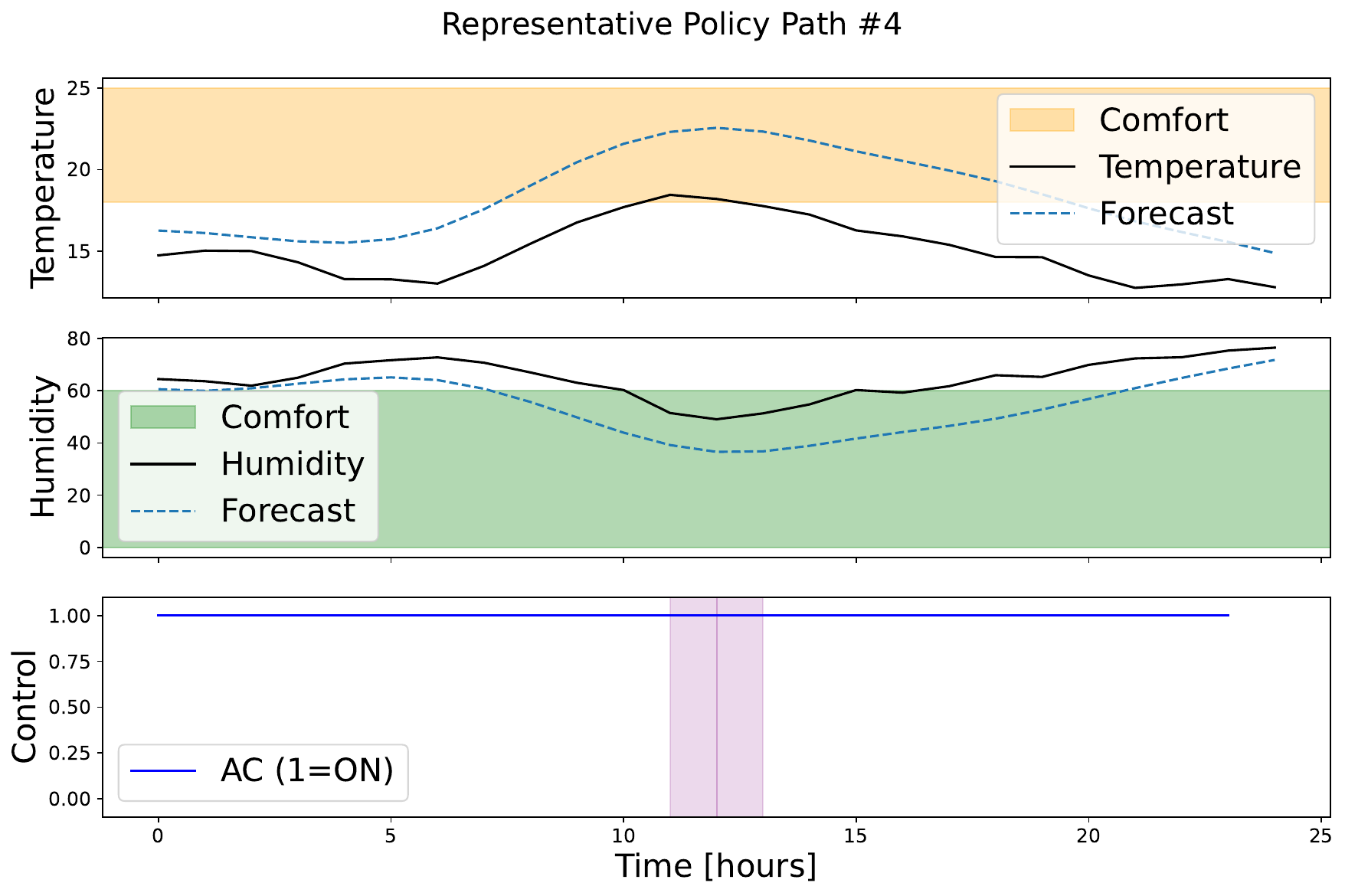}
\caption{Representative policy rollout under synthetic data and real forecasts. The policy activates in advance of potentially inadmissible temperature and humidity regimes.}
\label{fig:policy_path_4}
\end{figure}

Across the examples in \Cref{fig:policy_path_1,fig:policy_path_3,fig:policy_path_4}, the feedback control exhibits the expected qualitative structure. The system is kept off when the environmental process remains far from the failure set, while activation occurs early enough to compensate for the warm-up delay. The behaviour is therefore consistent with the policy heat maps above.

\subsection{Rolling-Horizon Deployment}

For real-data evaluation we use a rolling-horizon procedure. Every six hours a new forecast becomes available. At this time, a twelve-hour stochastic optimal control problem is solved, but only the first six hours of the resulting policy are implemented. The procedure is then repeated with the next forecast. This mitigates boundary effects and continuously incorporates newly available information.

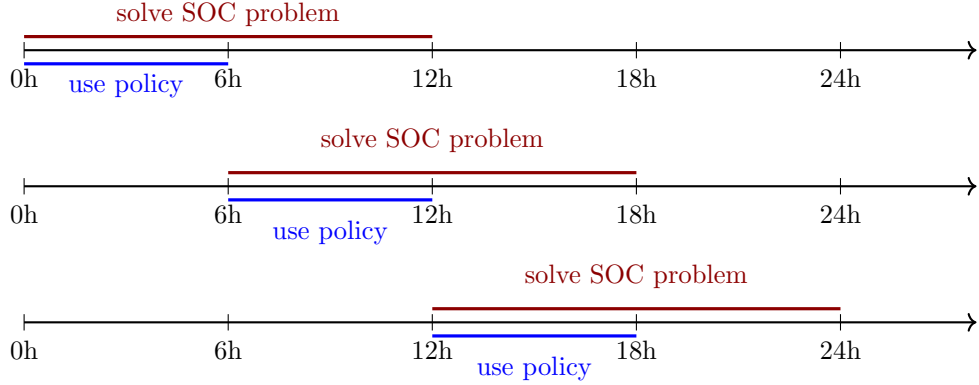
\begin{figure}[ht]
\centering
\begin{tikzpicture}[
    scale=0.9,
    solve/.style={very thick, darkred},
    apply/.style={very thick, blue},
    axis/.style={thick},
    every node/.style={font=\small}
]
\def\yone{0}
\def\ytwo{-2}
\def\ythree{-4}
\draw[axis,->] (0,\yone) -- (14,\yone);
\foreach \x/\t in {0/0,3/6,6/12,9/18,12/24}
{
    \draw (\x,\yone+0.12) -- (\x,\yone-0.12);
    \node[below] at (\x,\yone-0.12) {\t h};
}
\draw[solve] (0,\yone+0.2) -- (6,\yone+0.2);
\node[above,darkred] at (3,\yone+0.2) {solve SOC problem};
\draw[apply] (0,\yone-0.2) -- (3,\yone-0.2);
\node[below,blue] at (1.5,\yone-0.2) {use policy};
\draw[axis,->] (0,\ytwo) -- (14,\ytwo);
\foreach \x/\t in {0/0,3/6,6/12,9/18,12/24}
{
    \draw (\x,\ytwo+0.12) -- (\x,\ytwo-0.12);
    \node[below] at (\x,\ytwo-0.12) {\t h};
}
\draw[solve] (3,\ytwo+0.2) -- (9,\ytwo+0.2);
\node[above,darkred] at (6,\ytwo+0.35) {solve SOC problem};
\draw[apply] (3,\ytwo-0.2) -- (6,\ytwo-0.2);
\node[below,blue] at (4.5,\ytwo-0.35) {use policy};
\draw[axis,->] (0,\ythree) -- (14,\ythree);
\foreach \x/\t in {0/0,3/6,6/12,9/18,12/24}
{
    \draw (\x,\ythree+0.12) -- (\x,\ythree-0.12);
    \node[below] at (\x,\ythree-0.12) {\t h};
}
\draw[solve] (6,\ythree+0.2) -- (12,\ythree+0.2);
\node[above,darkred] at (9,\ythree+0.35) {solve SOC problem};
\draw[apply] (6,\ythree-0.2) -- (9,\ythree-0.2);
\node[below,blue] at (7.5,\ythree-0.35) {use policy};
\end{tikzpicture}
\caption{Rolling-horizon implementation. Every six hours the forecast is updated, a twelve-hour stochastic control problem is solved, and only the first six hours of the policy are implemented.}
\label{fig:rolling_horizon}
\end{figure}

\subsection{Real Weather Data Experiments}

We finally evaluate the rolling-horizon control strategy on real weather data
for Berlin over the full year 2024. This constitutes an out-of-sample
experiment: the stochastic forecast model is calibrated exclusively using data
from the years 2021--2023, while the subsequent policy evaluation is performed
on forecasts and observations from 2024. In contrast to the preceding
experiments, where environmental trajectories were sampled from the calibrated
stochastic model around a real forecast, the policy is now evaluated along the
actually observed temperature and relative-humidity trajectories. The
stochastic model is used only for computing the control policy: at every
rolling-horizon iteration, future environmental scenarios are sampled
conditionally on the currently available real weather forecast. The comfort region
is the same as in the previous section.

The controller is implemented in a rolling-horizon fashion. At each
optimization time, the corresponding real Berlin forecast is used as the
deterministic forecast trajectory \(\mu\). The forecast-tracking stochastic
model is parameterized using the calibration associated with the respective
calendar month, obtained from the 2021--2023 calibration data. Conditional on
the current forecast, \(8000\) trajectories are simulated from the calibrated
model and used to solve a \(12\)-hour stochastic control problem. The resulting
policy is applied to the real observed state only for the first \(6\) hours,
after which a new forecast is used and the optimization problem is solved
again. Both the observations and control decisions are taken on the original
hourly time grid. The continuation values are approximated using polynomial
regression of degree \(10\), and the clock state is carried forward between
successive rolling-horizon problems. The complete rolling-horizon procedure is
illustrated schematically in \Cref{fig:rolling_horizon}.

The resulting out-of-sample control trajectory is shown in
\Cref{fig:berlin_rollout}. The upper panels display the realized temperature
and relative humidity together with the corresponding comfort bounds, while
the lower panel shows the control selected by the rolling-horizon policy. The
control activity is concentrated primarily during periods in which violations
of the comfort region become relevant, while the controller remains inactive
for extended periods when the observed environmental conditions do not require
service.

To assess the economic performance of the policy, we compare it with an
always-on benchmark. In addition to the total and cumulative costs, we record
comfort violations and the number of switching events. The average daily cost
is also computed separately for each calendar month. The resulting monthly
comparison is shown in \Cref{fig:berlin_costs}. The savings are moderate but systematic within the considered ambient
service-demand model. Given the illustrative cost parameters and the
high penalty assigned to violations of the prescribed comfort region,
the resulting control policy tends to favor conservative operation.
These findings should be interpreted as model-based operating-cost
comparisons rather than validated indoor comfort improvements or
measured facility energy savings.

\begin{figure}[ht]
\centering
\includegraphics[width=0.82\textwidth]{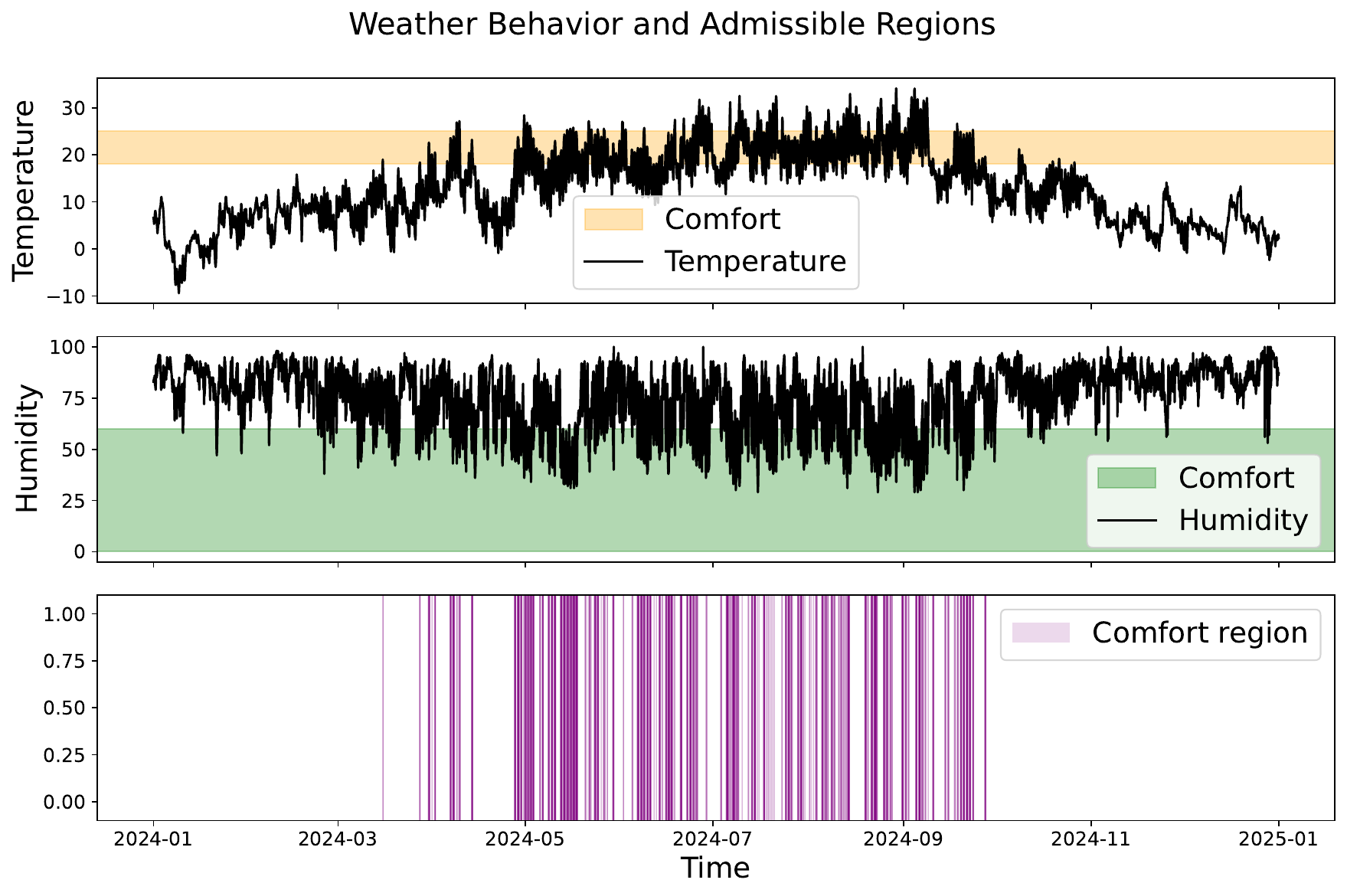}
\caption{Berlin weather data of 2024. \textbf{Upper panels:} These show temperature and humidity together with the comfort bounds. \textbf{Lower Panel:} This shows the time periods in which the data is in comfort regime in violet. We see that the time slices in which the system is in the comfort regime are scarce.}
\label{fig:berlin_rollout}
\end{figure}

\begin{figure}[ht]
\centering
\includegraphics[width=0.9\textwidth]{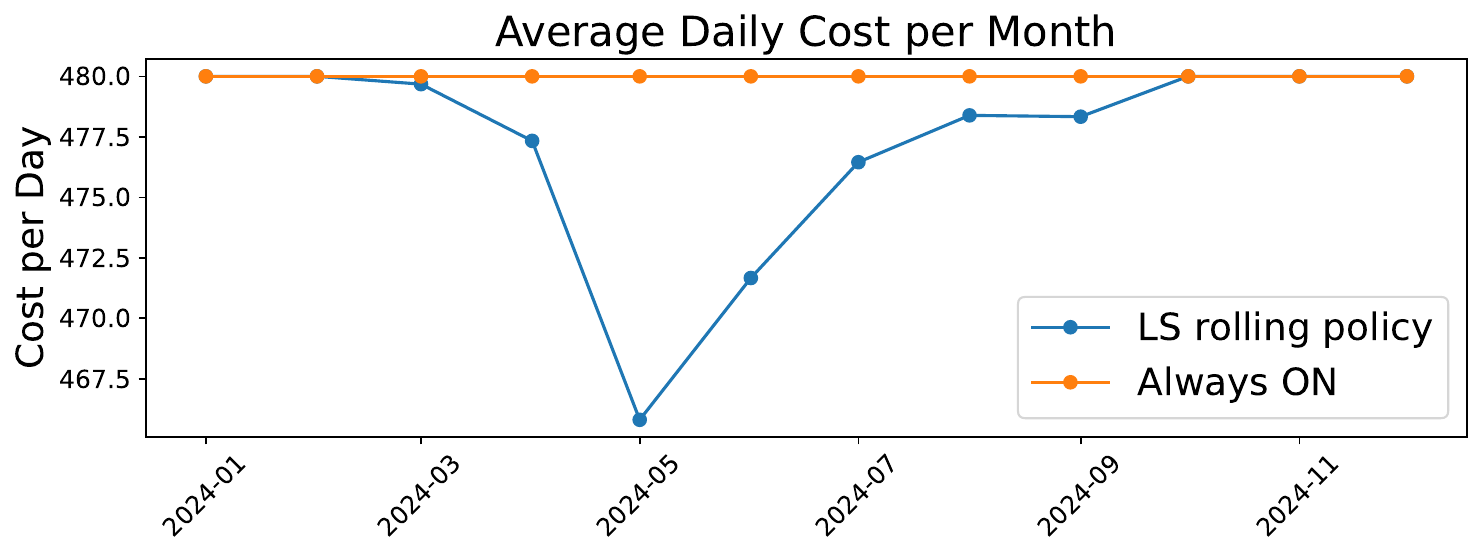}
\caption{Average daily cost per month for the rolling-horizon SOC policy and the permanently active benchmark.}
\label{fig:berlin_costs}
\end{figure}

\subsection{Primal--dual bracket}\label{primal-dual}

To assess the quality of the regression-based approximation, we complement
the out-of-sample policy evaluation by a primal--dual bound for the
time-discretized control problem. The purpose of this analysis is to
quantify the numerical approximation error in the solution of the
discrete-time dynamic program, rather than the time-discretization error
relative to the original continuous-time control problem.

Let
\[
    S_k=(X_k,\zeta_k)
\]
denote the augmented Markov state at time $t_k$, and let $V_k^\Delta$
be the value function of the discrete-time control problem. As
established in Section~\ref{gaussian-stage-cost}, the one-stage cost
$g_k(S_k,a)$ is the conditional expected cost accumulated over the
corresponding grid interval, including the switching penalty.
The Gaussian structure of the Ornstein--Uhlenbeck process allows this
cost to be reduced to a deterministic time integral. In particular,
$g_k(S_k,a)$ depends only on the current state and selected action.

An admissible policy $\widehat\pi$ obtained from the approximate
dynamic program immediately yields the primal upper bound
\[
    V_0^\Delta
    \leq
    U
    :=
    \mathbb E\bigl[J^\Delta(\widehat\pi)\bigr],
\]
where
\[
    J^\Delta(\widehat\pi)
    =
    \sum_{k=0}^{K-1}
    g_k(S_k,\widehat\pi_k).
\]
The expectation is evaluated out of sample, using independently
simulated trajectories of the exogenous process.

For a lower bound, we employ an information relaxation that allows
the controller to anticipate future realizations of the exogenous
process. The resulting information advantage is penalized by
martingale increments constructed from the approximate value
functions $\widehat V_k$ obtained by regression.

For a current state $(x,z)$ and action $a\in\{0,1\}$, let
\[
    z'=a(z+\Delta t_k)
\]
denote the next clock state. We define
\[
    \overline V_k(x,z')
    :=
    \mathbb E\left[
        \widehat V_{k+1}(X_{k+1},z')
        \mid X_k=x
    \right]
\]
and construct the martingale increment
\[
    \Delta M_{k+1}(a)
    =
    \widehat V_{k+1}(X_{k+1},z')
    -
    \overline V_k(X_k,z'),
    \qquad
    z'=a(\zeta_k+\Delta t_k).
\]
For every nonanticipative admissible policy, these increments have
conditional mean zero. Notice that the one-stage cost does not enter
the martingale increment: since $g_k(S_k,a)$ is already measurable
at time $t_k$, it would cancel under conditional centering.

Consequently, relaxing nonanticipativity and subtracting the
martingale increments from the pathwise costs gives the dual lower
bound
\[
    L
    :=
    \mathbb E\left[
        \inf_{\pi\in\Pi^{\mathrm{PI}}}
        \left\{
            J^\Delta(\pi)
            -
            \sum_{k=0}^{K-1}\Delta M_{k+1}(\pi_k)
        \right\}
    \right]
    \leq
    V_0^\Delta,
\]
where $\Pi^{\mathrm{PI}}$ denotes the set of policies with perfect
information about the future exogenous state trajectory. The state
dynamics, the deterministic evolution of the delay variable, and
the admissible action constraints remain unchanged; only the
nonanticipativity constraint is relaxed.

For each simulated exogenous trajectory $(X_0,\ldots,X_K)$,
the inner perfect-information problem is solved by backward
induction over the finite set of attainable clock states.
Setting $B_K(z)=0$, we compute
\[
\begin{aligned}
    B_k(z)
    =
    \min_{a\in\mathcal A_k(X_k,z)}
    \Bigl\{
        &g_k(X_k,z,a)
        +\overline V_k(X_k,z')\\
        &-\widehat V_{k+1}(X_{k+1},z')
        +B_{k+1}(z')
    \Bigr\},
    \qquad
    z'=a(z+\Delta t_k),
\end{aligned}
\]
where $\mathcal A_k(X_k,z)$ denotes the admissible actions at the
current state. The initial value $B_0(z_0)$ equals the pathwise
minimum appearing in the dual formulation, so that
\[
    L=\mathbb E[B_0(z_0)].
\]
Thus, the inner optimization is a finite backward recursion and
does not introduce an additional regression approximation.

In the numerical implementation, the conditional expectation
$\overline V_k$ is evaluated using tensor-product Gauss--Hermite
quadrature under the exact Gaussian transition law of the
Ornstein--Uhlenbeck process. In particular, writing
\[
    X_{k+1}\mid X_k=x
    \sim\mathcal N(\mu_k(x),\Sigma_k),
\]
the approximation takes the form
\[
    \overline V_k^{\mathrm{GH}}(x,z')
    =
    \sum_{j=1}^{N_{\mathrm{GH}}}
    w_j\,
    \widehat V_{k+1}
    \bigl(\mu_k(x)+A_k\xi_j,z'\bigr),
\]
where $A_kA_k^\top=\Sigma_k$ and
$(\xi_j,w_j)$ are the appropriately normalized
multidimensional Gauss--Hermite nodes and weights.
The approximate value function is evaluated as the minimum
of the two fitted action-value functions.

Since this minimum is generally not polynomial, Gauss--Hermite
quadrature is not guaranteed to evaluate the conditional
expectation exactly. Accordingly, the theoretical lower-bound
property holds for the exactly centred martingale increments,
whereas the computed dual estimate additionally contains
numerical quadrature error. We therefore interpret the
numerically obtained primal--dual gap as an a posteriori
diagnostic of approximation quality, rather than as a
rigorous error certificate in the absence of an explicit
quadrature-error bound.

To estimate the bracket, let
$X^{(1)},\ldots,X^{(N)}$ be independent out-of-sample
trajectories, and denote the corresponding primal policy costs
and pathwise dual values by $U^{(i)}$ and $L^{(i)}$.
We compute
\[
    \widehat U
    =
    \frac{1}{N}\sum_{i=1}^N U^{(i)},
    \qquad
    \widehat L
    =
    \frac{1}{N}\sum_{i=1}^N L^{(i)}.
\]
The estimated absolute and relative gaps are defined by
\[
    \widehat G
    =
    \widehat U-\widehat L,
    \qquad
    \widehat G_{\mathrm{rel}}
    =
    \frac{\widehat G}{|\widehat U|}.
\]
Since the primal and dual quantities are evaluated on the
same simulated trajectories, we estimate the standard error
of the gap from the paired observations
\[
    G^{(i)}=U^{(i)}-L^{(i)},
\]
namely
\[
    \widehat{\mathrm{SE}}(\widehat G)
    =
    \sqrt{
        \frac{1}{N(N-1)}
        \sum_{i=1}^N
        \left(G^{(i)}-\widehat G\right)^2
    }.
\]
The standard errors of $\widehat U$ and $\widehat L$ are
computed analogously from their respective pathwise
observations.

With exact conditional centering, information-relaxation
duality gives
\[
    L\leq V_0^\Delta\leq U.
\]
The reported numerical estimates approximate these bounds
subject to Monte Carlo sampling error and, for the dual
estimate, Gauss--Hermite quadrature error. A small estimated
primal--dual gap therefore provides evidence of a small
numerical optimality gap for the discrete-time problem, but
does not by itself measure the time-discretization error or
out-of-model policy performance.

\subsection{Synthetic Forecast Experiments}

Finally, we investigate the sensitivity of the methodology to the quality of the available forecast information. To this end, we replace operational forecasts with synthetic forecasts constructed from window-averaged realized weather observations. Since these synthetic forecasts are derived retrospectively from observed data, they may incorporate information that would not have been available at the corresponding decision times. Consequently, this experiment should be interpreted as a retrospective sensitivity analysis rather than an evaluation of a practically deployable forecasting system.

\begin{table}[ht]
\centering
\begin{tabular}{c|c|c}
 & Operational Forecast & Synthetic Forecast \\
\hline
Yearly Average Cost/Day (Always On) & 480 USD & 480 USD \\
\hline
Yearly Average Cost/Day (SOC Policy) & 478 USD & 464 USD \\
\hline
Cost Savings & $\sim 0.41\%$ & $\sim 3.33\%$ \\
\hline
\end{tabular}
\caption{Annual average daily operating costs under operational and retrospectively constructed synthetic forecasts.}
\label{tab:forecast_costs}
\end{table}

As shown in \Cref{tab:forecast_costs}, the estimated savings increase from approximately $0.41\%$ under operational forecasts to $3.33\%$ under retrospectively constructed synthetic forecasts. This indicates that the performance of the proposed methodology is sensitive to the forecast information supplied to the controller and suggests that improvements in forecast quality may offer additional opportunities for operational cost reduction. However, the experiment does not isolate forecast uncertainty from other sources of approximation error, including regression, time discretization, receding-horizon optimization, and environmental-model inaccuracies. Accordingly, these results should not be interpreted as establishing that forecast error is the dominant limitation of the proposed methodology.

\section{Conclusion}

We developed a stochastic optimal control framework for air-conditioning systems 
with delayed activation.  
The resulting problem was formulated in weak open-loop form, with 
optimization taken over admissible stochastic bases as well as controls.

We introduced a regression-based dynamic programming 
approximation on nested grids and proved convergence of the 
idealized discretized value functions to the continuous-time value function.

On the computational side, the method combines a forecast-tracking Ornstein--Uhlenbeck model, monthly maximum-likelihood calibration, regression-based dynamic programming, and rolling-horizon deployment. The numerical experiments show that the learned feedback controls exhibit anticipatory switching behaviour: the system is activated before the comfort region is violated in order to compensate for the warm-up delay.

The real-data Berlin experiments demonstrate systematic cost reductions relative to a permanently active climate-control system, while maintaining conservative behaviour in the presence of large violation penalties. The synthetic forecast experiments indicate that substantially larger savings are possible when forecast quality improves. This suggests that, in the present implementation, the practical performance is limited more by forecast accuracy than by the control methodology itself.

\section*{Acknowledgments}

C.B. gratefully acknowledges the support of the German research council DFG through the CRC TRR 388.

\printbibliography

@article{kallenberg_sde,
 ISSN = {00911798, 2168894X},
 URL = {http://www.jstor.org/stable/2244838},
 author = {Olav Kallenberg},
 journal = {The Annals of Probability},
 number = {1},
 pages = {196--205},
 publisher = {Institute of Mathematical Statistics},
 title = {On the Existence of Universal Functional Solutions to Classical SDE'S},
 urldate = {2026-08-16},
 volume = {24},
 year = {1996}
}

@article{benth_temperature_2005,
  author  = {Fred Espen Benth and Jurate \v{S}altyt\.{e}-Benth},
  title   = {Stochastic Modeling of Temperature Variations with a View Towards Weather Derivatives},
  journal = {Applied Mathematical Finance},
  year    = {2005},
  volume  = {12},
  number  = {1},
  pages   = {53--85},
  doi     = {10.1080/1350486042000271638}
}

@book{kushner_dupuis_2001,
  author    = {Harold J. Kushner and Paul G. Dupuis},
  title     = {Numerical Methods for Stochastic Control Problems in Continuous Time},
  publisher = {Springer},
  address   = {New York},
  edition   = {2},
  year      = {2001},
  series    = {Applications of Mathematics},
  volume    = {24},
  doi       = {10.1007/978-1-4613-0007-6},
  url       = {https://doi.org/10.1007/978-1-4613-0007-6}
}

@article{ludkovski_regression_impulse,
     author = {Mike Ludkovski},
     title = {Regression {Monte} {Carlo} for {Impulse} {Control}},
     journal = {MathematicS In Action},
     pages = {73--90},
     year = {2022},
     publisher = {Soci\'et\'e de Math\'ematiques Appliqu\'ees et Industrielles},
     volume = {11},
     number = {1},
     doi = {10.5802/msia.18},
     language = {en},
     url = {https://msia.centre-mersenne.org/articles/10.5802/msia.18/}
}

@article{palczewski_stettner_delay,
author = {Palczewski, Jan and Stettner, \L{}ukasz},
title = {Finite Horizon Optimal Stopping of Time-Discontinuous Functionals with Applications to Impulse Control with Delay},
journal = {SIAM Journal on Control and Optimization},
volume = {48},
number = {8},
pages = {4874-4909},
year = {2010},
doi = {10.1137/080737848},

URL = { 
    
        https://doi.org/10.1137/080737848
    
    

},
eprint = { 
    
        https://doi.org/10.1137/080737848
    
    

}
}

@article{longstaff,
    author = {Longstaff, Francis A. and Schwartz, Eduardo S.},
    title = {Valuing American Options by Simulation: A Simple Least-Squares Approach},
    journal = {The Review of Financial Studies},
    volume = {14},
    number = {1},
    pages = {113-147},
    year = {2001},
    month = {01},
    issn = {0893-9454},
    doi = {10.1093/rfs/14.1.113},
    url = {https://doi.org/10.1093/rfs/14.1.113},
    eprint = {https://academic.oup.com/rfs/article-pdf/14/1/113/24432078/113.pdf},
}

@article{oksendal_sulem_delayed_reaction,
  author    = {Bernt {\O}ksendal and Agn{\`e}s Sulem},
  title     = {Optimal Stochastic Impulse Control with Delayed Reaction},
  journal   = {Applied Mathematics and Optimization},
  year      = {2008},
  volume    = {58},
  number    = {2},
  pages     = {243--255},
  doi       = {10.1007/s00245-007-9034-5},
  url       = {https://doi.org/10.1007/s00245-007-9034-5}
}

@article{kumar_hvac_mpc,
title = {Stochastic model predictive control for central HVAC plants},
journal = {Journal of Process Control},
volume = {90},
pages = {1-17},
year = {2020},
issn = {0959-1524},
doi = {https://doi.org/10.1016/j.jprocont.2020.03.015},
url = {https://www.sciencedirect.com/science/article/pii/S0959152420301943},
author = {Ranjeet Kumar and Michael J. Wenzel and Mohammad N. ElBsat and Michael J. Risbeck and Kirk H. Drees and Victor M. Zavala},
}

@misc{yang_hvac_soc,
      title={Stochastic Optimal Control of HVAC system for Energy-efficient Buildings}, 
      author={Yu Yang and Guoqiang Hu and Costas J. Spanos},
      year={2021},
      eprint={1911.00840},
      archivePrefix={arXiv},
      primaryClass={eess.SY},
      url={https://arxiv.org/abs/1911.00840}, 
}

@article{zeng_hvac_review,
AUTHOR = {Zeng, Xinran and Li, Chunhui and Li, Xiaoying and Mao, Chennan and Li, Zhengwei and Li, Zhenhai},
TITLE = {Energy Efficiency Optimization of Air Conditioning Systems Towards Low-Carbon Cleanrooms: Review and Future Perspectives},
JOURNAL = {Energies},
VOLUME = {18},
YEAR = {2025},
NUMBER = {13},
ARTICLE-NUMBER = {3538},
URL = {https://www.mdpi.com/1996-1073/18/13/3538},
ISSN = {1996-1073},
DOI = {10.3390/en18133538}
}

@article{yin_cleanroom_energy,
title = {Influence of thermal energy storage and heat pump parametrization for demand-side-management in a nearly-zero-energy-building using model predictive control},
journal = {Energy and Buildings},
volume = {226},
pages = {110364},
year = {2020},
issn = {0378-7788},
doi = {https://doi.org/10.1016/j.enbuild.2020.110364},
url = {https://www.sciencedirect.com/science/article/pii/S0378778820313086},
author = {Steffen Bechtel and Sasan Rafii-Tabrizi and Frank Scholzen and Jean-Régis Hadji-Minaglou and Stefan Maas},
}

@article{ni_semiconductor_hvac,
title = "Optimizing HVAC systems for semiconductor fabrication: A data-intensive framework for energy efficiency and sustainability",
author = "Ni, \{Hsiao Ping\} and Chong, \{Wai Oswald\} and Chou, \{Jui Sheng\}",
note = "Publisher Copyright: {\textcopyright} 2024 Elsevier Ltd",
year = "2024",
month = jul,
day = "15",
doi = "10.1016/j.jobe.2024.109397",
language = "English (US)",
volume = "89",
journal = "Journal of Building Engineering",
issn = "2352-7102",
publisher = "Elsevier B.V.",
}

@article{perez_lombard_building_energy,
title = {A review on buildings energy consumption information},
journal = {Energy and Buildings},
volume = {40},
number = {3},
pages = {394-398},
year = {2008},
issn = {0378-7788},
doi = {https://doi.org/10.1016/j.enbuild.2007.03.007},
url = {https://www.sciencedirect.com/science/article/pii/S0378778807001016},
author = {Luis Pérez-Lombard and José Ortiz and Christine Pout},
}

@book{bertsekas,
  title={Stochastic Optimal Control: The Discrete-Time Case},
  author={Bertsekas, D. and Shreve, S.E.},
  isbn={9781886529038},
  lccn={96080191},
  series={Athena scientific optimization and computation series},
  url={https://books.google.com.sa/books?id=sRHEAAAACAAJ},
  year={1996},
  publisher={Athena Scientific}
}

@article{renzo,
  author    = {Caballero, Renzo and Kebaier, Ahmed and Scavino, Marco and Tempone, Ra{\'u}l},
  title     = {Quantifying uncertainty with a derivative tracking SDE model and application to wind power forecast data},
  journal   = {Statistics and Computing},
  year      = {2021},
  volume    = {31},
  number    = {5},
  pages     = {64},
  doi       = {10.1007/s11222-021-10040-8},
  publisher = {Springer}
}

@online{openweather,
  author  = {{OpenWeather}},
  title   = {OpenWeather Weather API},
  year    = {2026},
  url     = {https://openweathermap.org/api},
  urldate = {2026-06-09}
}

@article{regression,
author = {Belomestny, Denis and Kolodko, Anastasia and Schoenmakers, John},
title = {Regression Methods for Stochastic Control Problems and Their Convergence Analysis},
journal = {SIAM Journal on Control and Optimization},
volume = {48},
number = {5},
pages = {3562-3588},
year = {2010},
doi = {10.1137/090752651}
}

@book{kallenberg,
  title={Foundations of Modern Probability},
  author={Kallenberg, O.},
  isbn={9780387953137},
  lccn={2001032816},
  series={Probability and Its Applications},
  url={https://books.google.com.sa/books?id=L6fhXh13OyMC},
  year={2002},
  publisher={Springer New York}
}

@article{howard,
  title={Dynamic programming and markov processes.},
  author={Howard, Ronald A},
  year={1960},
  publisher={John Wiley}
}

\appendix

\section{Proof of \Cref{feedback}}\label{app:proofprop}
Let \(\alpha\) be a given control. Without loss of generality, we may assume
that, on each interval \([t_k,t_{k+1}]\), \(X\) is given by the measurable
SDE solution map from \cite{kallenberg_sde} used in the definition of
\(\eta^{t_k}\), since this representation agrees with the original state
process outside a set of probability zero.
Then we can define a policy
\[
\pi^\alpha
=
\bigl(
\mu^{\alpha,t_0},
\ldots,
\mu^{\alpha,t_{K-1}}
\bigr)
\]
through the disintegration
\[
\mathcal{L}
\bigl(
X_{t_k},
\zeta_{t_k},
\alpha_{t_k}
\bigr)
=
\mathcal{L}
\bigl(
X_{t_k},
\zeta_{t_k}
\bigr)
\otimes
\mu^{\alpha,t_k}.
\]
Let us first assume that
\begin{align}\label{eqval}
J(0,x_0,z_0;\alpha)
=
J^D_{\pi^\alpha}(x_0,z_0).
\end{align}
Then \Cref{feedback} would yield an optimal feedback policy
\(\pi^f\), where
\[
\mu^{f,t_k}_{x,z}
=
\delta_{u_{t_k}(x,z)},
\]
which gives an optimal feedback control through
\[
\alpha^f_{t_k}
=
u_{t_k}
\bigl(
X^f_{t_k},
\zeta^f_{t_k}
\bigr).
\]
Hence, it is enough to prove \Cref{eqval}. To this end, observe that for
\(\alpha\in\mathcal A^D\), the continuous-time cost can equivalently be
written as a discrete-time cost. Indeed, by the tower property and the
definition of \(g_{t_k}\),
\begin{align*}
J(0,x_0,z_0;\alpha)
&=
\sum_{k=0}^{K-1}
\mathbb E
\left[
g_{t_k}
\bigl(
X_{t_k},
\zeta_{t_k},
\alpha_{t_k}
\bigr)
\right].
\end{align*}
Here, on each interval \([t_k,t_{k+1}]\), the conditional expectation in
\(g_{t_k}\) integrates out the Brownian path on that interval. Thus, although
the original running cost depends on the continuous state trajectory, its
expectation over each control interval depends only on the joint distribution
of
\[
\bigl(
X_{t_k},
\zeta_{t_k},
\alpha_{t_k}
\bigr).
\]
Consequently, by linearity of expectation, the cost functional depends only on
these joint distributions at the individual control dates. Let \(q^{t_k}\) be the \(k\)-th marginal of
\[
r_K\bigl(\pi^\alpha,(x_0,z_0)\bigr),
\]
i.e.
\begin{align*}
q^{t_0}\bigl(\pi^\alpha,(x_0,z_0)\bigr)
&=
\delta_{x_0,z_0}\otimes\mu^{\alpha,t_0},
\\
q^{t_k}\bigl(\pi^\alpha,(x_0,z_0)\bigr)
&=
\bigl(
q^{t_{k-1}}\eta^{t_{k-1}}
\bigr)
\otimes
\mu^{\alpha,t_k}.
\end{align*}
We are done if we manage to show that
\begin{align}\label{induction}
\mathcal{L}
\bigl(
X_{t_k},
\zeta_{t_k},
\alpha_{t_k}
\bigr)
=
q^{t_k}.
\end{align}
We proceed by induction. The case \(k=0\) follows by definition, as
\[
\mathcal{L}
\bigl(
X_{t_0},
\zeta_{t_0},
\alpha_{t_0}
\bigr)
=
\mathcal{L}
\bigl(
X_{t_0},
\zeta_{t_0}
\bigr)
\otimes
\mu^{\alpha,t_0}
=
\delta_{(x_0,z_0)}
\otimes
\mu^{\alpha,t_0}.
\]
Assuming \Cref{induction} for \(k-1\), we obtain the second equality by the
equation of state of \((X,\zeta)\) and the third by the induction hypothesis:
\begin{align*}
\mathcal{L}
\bigl(
X_{t_k},
\zeta_{t_k},
\alpha_{t_k}
\bigr)
&=
\mathcal{L}
\bigl(
X_{t_k},
\zeta_{t_k}
\bigr)
\otimes
\mu^{\alpha,t_k}
\\
&=
\left(
\mathcal{L}
\bigl(
X_{t_{k-1}},
\zeta_{t_{k-1}},
\alpha_{t_{k-1}}
\bigr)
\eta^{t_{k-1}}
\right)
\otimes
\mu^{\alpha,t_k}
\\
&=
\bigl(
q^{t_{k-1}}\eta^{t_{k-1}}
\bigr)
\otimes
\mu^{\alpha,t_k}
\\
&=
q^{t_k}.
\end{align*}

\end{document}